\documentclass[11pt]{article}

\usepackage[a4paper,margin=1in]{geometry}
\usepackage[british]{babel}
\usepackage{amsmath,amssymb,amsthm,mathtools,amscd}
\usepackage{mathrsfs}
\usepackage{bm}
\usepackage{booktabs,tabularx,array}
\usepackage{enumitem}
\usepackage{microtype}
\usepackage{xcolor}
\usepackage[colorlinks=true,allcolors=black]{hyperref}
\usepackage[nameinlink,noabbrev]{cleveref}
\usepackage{orcidlink}

\newtheorem{theorem}{Theorem}[section]
\newtheorem{proposition}[theorem]{Proposition}
\newtheorem{lemma}[theorem]{Lemma}
\newtheorem{corollary}[theorem]{Corollary}
\theoremstyle{definition}
\newtheorem{definition}[theorem]{Definition}
\newtheorem{assumption}[theorem]{Assumption}
\theoremstyle{remark}
\newtheorem{remark}[theorem]{Remark}

\numberwithin{equation}{section}

\newcommand{\eps}{\varepsilon}
\newcommand{\del}{\delta}
\newcommand{\R}{\mathbb R}
\newcommand{\N}{\mathbb N}

\newcommand{\Hcal}{\mathcal H}
\newcommand{\Xcal}{\mathcal X}
\newcommand{\Mcal}{\mathcal M}
\newcommand{\Acal}{\mathcal A}
\newcommand{\Bcal}{\mathcal B}
\newcommand{\Dcal}{\mathcal D}
\newcommand{\Ecal}{\mathcal E}

\newcommand{\Ucal}{\mathcal U}
\newcommand{\Rcal}{\mathcal R}

\newcommand{\norm}[1]{\left\lVert #1\right\rVert}
\newcommand{\abs}[1]{\left|#1\right|}
\newcommand{\set}[1]{\left\{#1\right\}}

\newcommand{\dist}{\operatorname{dist}}
\newcommand{\Lip}{\operatorname{Lip}}
\newcommand{\Id}{\operatorname{Id}}

\newcommand{\rank}{\operatorname{rank}}
\newcommand{\dom}{\operatorname{D}}
\newcommand{\e}{\mathrm e}

\newcommand\blfootnote[1]{%
	\begingroup
	\renewcommand\thefootnote{}\footnote{#1}%
	\addtocounter{footnote}{-1}%
	\endgroup
}

\title{Renormalized Invariant Manifolds}
\author{Christian Kuehn \orcidlink{0000-0002-7063-6173}$^{1}$
\& Jan-Eric Sulzbach \orcidlink{ 0000-0002-9446-2366}$^{2}$
}

\date{
	\small{$^1$\textit{Technical University of Munich, School of Computation, Information and Technology, \\ Department of Mathematics, Boltzmannstraße 3, 85748 Garching, Germany} \\
    $^2$\textit{Leiden University, Mathematical Institute, Einsteinweg 55, 2333 CC Leiden, Netherlands}}
}

\begin{document}
\maketitle

\begin{abstract}
    Invariant manifold theory for partial differential equations (PDEs) is technically challenging as we often lack spectral gaps. 
    Motivated by recent progress using renormalization in various areas of mathematics, we introduce a spectral renormalization method for semilinear parabolic equations that creates an artificial spectral gap at high frequencies while leaving the low modes unchanged. 
    For the resulting family of renormalized equations, classical inertial-manifold theory yields finite-dimensional invariant manifolds whose dimension diverges as the renormalization parameter tends to zero. 
    We show that the reduced dynamics on these manifolds approximate the original semiflow through an approximate semi-conjugacy, with an $O(\eps^2)$ error on bounded finite-time intervals.
    Under global incremental dissipativity, the semi-conjugacy is uniform in time and yields an $O(\eps^2)$ Hausdorff estimate for the global attractors, if they exist. 
    We also establish further dynamical properties of the spectral renormalization such as persistence of hyperbolic equilibria and their Morse indices.
    We illustrate the construction in settings with and without a natural spectral gap and embed it into a two-parameter fast-slow PDE framework, where the artificial gap enables us to construct slow manifolds for systems on bounded domains with arbitrary dimension. 
    Thus our new approach provides a finite-dimensional dynamical approximation even in regimes where an exact inertial manifold for the original equation is not available. In particular, this provides a practical balance between fully invariant manifolds and direct estimates. 
\end{abstract}

\vspace{2mm}
	
	
	\vspace{1mm}
	
	\vspace{3mm}
%

\blfootnote{\textcolor{white}{.}\\[-2.5mm]
	\hspace{-5.4mm}  \; ckuehn@ma.tum.de (Christian Kuehn)  \\[0.8mm]
	 \; j.e.sulzbach@math.leidenuniv.nl (Jan-Eric Sulzbach)
	}

\tableofcontents


\section{Introduction}
\label{sec:introduction}

Geometric singular perturbation theory provides a powerful framework for reducing dynamical systems with separated time scales.  
In finite dimensions, Fenichel theory shows that a normally hyperbolic critical manifold persists as a locally invariant slow manifold and that the dynamics of the full system can be approximated by a reduced flow on this manifold \cite{Fenichel1979,Jones1995}.  
In recent years, parts of this theory have been extended to infinite-dimensional evolution equations and, in particular, to fast-slow PDEs; see, for example, \cite{BatesLuZeng1998} and \cite{HummelKuehn2022,KuehnetAL2024,KuehnSulzbach2025} and the references therein.
A distinctive difficulty in the infinite-dimensional setting is that the slow variable itself contains arbitrarily high spatial frequencies.
To construct a  slow manifold one therefore has to separate the genuinely slow modes from an infinite-dimensional tail which can be treated as an additional stable fast component.

For parabolic problems this separation is closely tied to the spectrum of the underlying elliptic operator.  
To illustrate the issue, consider a semilinear equation
\begin{equation}
    \partial_t v+Av=F(v)
    \label{eq:original-equation-intro}
\end{equation}
on a Hilbert space $\Hcal$, where $A$ is a nonnegative self-adjoint operator with compact resolvent.  
If
\[
    0\leq\lambda_1\leq\lambda_2\leq\cdots,
    \qquad
    \lambda_j\to\infty,
\]
are the eigenvalues of $A$, a standard Lyapunov-Perron construction of an inertial manifold requires a sufficiently large separation between the last retained eigenvalue and the first discarded one.
If A is a one-dimensional elliptic operator this is often available naturally.  
For example, the eigenvalues of the Laplacian on an interval satisfy $\lambda_j\sim cj^2$, so that $\lambda_{j+1}-\lambda_j\to\infty$.  
In dimension two or higher, however, the spectral geometry is fundamentally different.  
Weyl's law no longer yields growing consecutive gaps, and the classical large-gap condition cannot in general be verified.  
This is one of the principal obstructions to applying infinite-dimensional slow-manifold constructions to reaction-diffusion systems on higher-dimensional bounded domains.

The same obstruction is classical in the theory of inertial manifolds.  
An inertial manifold is a finite-dimensional invariant Lipschitz manifold which contains the global attractor and exponentially tracks all trajectories.  
When it exists, it gives an exact finite-dimensional realization of the long-time dynamics of a dissipative PDE; see, for example, \cite{FoiasSellTemam1988,ConstantinEtAl1989,Zelik2014,KostiankoZelik2017}.
The standard construction again relies on a sufficiently large spectral gap.  
The absence of such gaps is not merely a technical shortcoming of a particular proof.  
There are abstract semilinear parabolic equations with uniformly bounded consecutive spectral gaps for which no finite-dimensional Lipschitz inertial manifold can exist; see the counterexamples in \cite{Zelik2014}.  
On the other hand, additional structure can sometimes replace the ordinary gap condition.  The most important example is the spatial-averaging method of Mallet-Paret and Sell for
scalar reaction-diffusion equations in higher space dimensions, which has recently been revisited and extended in \cite{MalletParetSell1988,KostiankoEtAl2022}.  
This mechanism, however, exploits a special scalar structure of the derivative of the Nemytskii operator and is not directly robust for genuinely coupled systems.

The purpose of the present work is to investigate a different approach and change the perspective to balance theoretical requirements of fully invariant objects with practical tasks of studying the dynamics of PDEs via approximations.  
Rather than asking the original equation to possess a large spectral gap, we introduce a controlled modification of the high-frequency linear dynamics which creates one; this idea is related to renormalization, e.g., in stochastic PDEs~\cite{hairer2014theory}. 
Fixing $\mu>0$, we choose the spectral cutoff
\begin{equation}
    \Lambda_\eps:=\mu\eps^{-2},
    \qquad
    P_\eps:=\mathbf 1_{[0,\Lambda_\eps]}(A),
    \qquad
    Q_\eps:=I-P_\eps,
\end{equation}
and replace $A$ by
\begin{equation}
    A_\eps
    :=
    A+\eps^{-1}Q_\eps.
\end{equation}
Thus the modes below the cutoff are left completely unchanged, whereas every discarded mode receives an additional damping of size $\eps^{-1}$.  
The renormalized equation is
\begin{equation}
    \partial_t v^\eps+A_\eps v^\eps=F(v^\eps).
    \label{eq:intro-renormalized-equation}
\end{equation}
If $\beta_\eps$ denotes the upper spectral edge of the retained subspace and $\alpha_\eps$ the lower spectral edge of the modified high modes, then
\[
    \alpha_\eps-\beta_\eps
    \geq\eps^{-1}.
\]
Hence the classical inertial-manifold gap condition is automatically satisfied for all sufficiently small $\eps$, independently of the natural consecutive gaps of $A$.

This construction should not be interpreted as asserting that the original
equation possesses an inertial manifold, nor is it simply a Galerkin truncation, which discards the modes in $Q_\eps\Hcal$ and evolves only the projected state $p\in P_\eps\Hcal$.
By contrast, the renormalized problem remains an evolution equation on the
full phase space $\Hcal$.
Hence for each fixed $\eps>0$, the modified equation possesses an exact finite-dimensional
inertial manifold
\[
    \Mcal_\eps
    =
    \{
       p+h_\eps(p):
       p\in P_\eps\Hcal
    \},
\]
but its dimension tends to infinity as $\eps\to0$.  
At the same time, the inserted gap is moved to increasingly high spectral frequencies.

The central question is therefore not whether $\Mcal_\eps$ converges to a fixed invariant manifold of the original PDE, but whether the finite-dimensional dynamics carried by $\Mcal_\eps$ provide a faithful approximation of the dynamically relevant part of \eqref{eq:original-equation-intro} as the spectral resolution increases.

Our main tool for answering this question is an approximate semi-conjugacy. In particular, instead of asking exact topological equivalence or conjugacy between the original and the renormalized PDE (which evidently will not hold in general), we want to control an approximation error to a conjugacy. Controlling this error is then often enough to analyze the dynamics rigorously using the invariant manifold of the renormalized PDE.  Let us denote the original semiflow by $S(t)$ and the flow on $\Mcal_\eps$ by $S_\eps^\Mcal(t)$, we estimate the defect in the commutative diagram
\[
    R_\eps S(t) -   S_\eps^\Mcal(t)R_\eps,
\]
where the graph retraction is given by
\begin{equation*}
    R_\eps v
    :=
    P_\eps v+h_\eps(P_\eps v)
    \in\Mcal_\eps.
\end{equation*}
The many-to-one character of $R_\eps$ is essential: unresolved high-frequency information is discarded rather than encoded into a finite-dimensional conjugacy.  
Nevertheless, its influence on the retained dynamics is small.  More precisely, for every bounded set $B\subset\Hcal$ and every fixed $T>0$, we are going to show that under suitable assumptions
\begin{equation}
    \sup_{v_0\in B}
    \sup_{0\leq t\leq T}
    \norm{
       R_\eps S(t)v_0
       -
       S_\eps^\Mcal(t)R_\eps v_0
    }_{\Hcal}
    \leq
    C_{B,T}\eps^2
    \label{eq:intro-finite-semiconjugacy}
\end{equation}
for some constant $C_{B,T}>0$ and some fixed finite time $T>0$. This estimate compares the reduced trajectory with the retracted original state.
For general initial data the discarded component itself need not be small at $t=0$, but it decays on the parabolic time scale $O(\eps^2)$.  
After an initial layer of length \(O\bigl(\eps^2|\log\eps|\bigr)\), the reduced trajectory therefore approximates the original state directly with error $O(\eps^2)$ on bounded finite time intervals.

The finite-time result is the basic level of comparison and requires no asymptotic stability of the retained dynamics.  
Stronger dynamical conclusions are obtained under additional hypotheses.  
If the reduced dynamics are incrementally contractive on a region entered by both trajectories in a common finite time, the comparison extends uniformly to all forward times. 
This yields a basinwise result when the local contraction and uniform entrance hypotheses are verified, as in the synchronized-sink example below. The global contraction hypothesis used for the attractor estimate is stronger and forces the
global attractors to be single equilibria.  
At a finer level, the artificial perturbation is small in suitable graph norms on regular states.  
This permits continuation of nondegenerate equilibria and, under hyperbolicity, preservation of their Morse indices.  
In Morse--Smale regimes, finite-time closeness may instead be combined with shadowing and structural-stability theory to relate complete orbits; compare the recent infinite-dimensional shadowing results of \cite{ArrietaCarvalhoTakaessu2026}.  
Thus the conclusions form a hierarchy: increasing dynamical information is obtained as
correspondingly stronger structural assumptions are imposed.

The viewpoint of this paper is therefore intermediate between two classical extremes.  
We do not require the original PDE to admit a single exact finite-dimensional invariant manifold, but neither do we merely truncate its high modes.  
Instead, we construct a family of nearby equations for which the unresolved modes are exactly slaved to the retained variables, and then quantify how much dynamical information survives when this artificial modification is removed.
The distinguishing feature is therefore the combination of an exact invariant reduction of the modified PDE with a controlled dynamical relation to the unmodified equation.

\paragraph{Organization of the paper.}
In Section \ref{sec:functional-setting} we introduce the spectral renormalization and establish the basic high-mode estimates.
In Section \ref{sec:inertial-manifold} we apply the classical spectral-gap theorem to construct the inertial manifold of the renormalized equation.
Section \ref{sec:approximate-semiconjugacy} develops the approximate semi-conjugacy, including finite-time estimates, direct post-initial-layer shadowing, and uniform-in-time comparison in dynamically stable regions.
Section \ref{sec:three-benchmarks} compares the method with intrinsic spectral gaps in the one-dimensional setting, spatial averaging, and a no-inertial-manifold example.
Section \ref{sec:dynamical-consequences} further develops the dynamical consequences of the renormalization, including attractor comparison and persistence of hyperbolic equilibria and their Morse indices.  
The coupled reaction-diffusion example in Section \ref{sec:coupled-AC-benchmark} illustrates these results for pattern-forming dynamics on a three-dimensional bounded domain.  
Section \ref{sec:fast-slow} embeds the construction into a two-parameter fast-slow PDE framework and constructs the corresponding slow manifold.
We conclude with a discussion of alternative spectral modifications, the relation with shadowing and structural stability, and connections with approximate and computational inertial manifolds.


\section{Functional Setting and Assumptions}
\label{sec:functional-setting}

In this section we introduce the phase space, the nonlinearity, and the spectral modification, and record the estimates used throughout the paper.

\begin{assumption}[Linear operator]
\label{ass:operator}
Let $\Hcal$ be an infinite-dimensional real Hilbert space and let $A:\dom(A)\subset\Hcal\to\Hcal$ be a nonnegative self-adjoint operator with compact resolvent.
Denote its eigenvalues, repeated according to multiplicity, by
\[
    0\leq\lambda_1\leq\lambda_2\leq\cdots,
    \qquad \lambda_j\to\infty.
\]
Then, $-A$ generates an analytic contraction semigroup $(\e^{-tA})_{t\geq0}$ on $\Hcal$.
\end{assumption}

For $r\geq0$, define the Hilbert scale
\begin{equation*}
    \Hcal^r:=\dom((I+A)^{r/2}),
    \qquad
    \norm{u}_r:=\norm{(I+A)^{r/2}u}_{\Hcal}.
\end{equation*}
For standard second-order elliptic realizations on smooth bounded domains, this scale describes Sobolev regularity together with the boundary compatibility conditions imposed by the operator and its powers.

\begin{assumption}[Nonlinearity]
\label{ass:nonlinearity}
The map $F:\Hcal\to\Hcal$ is globally bounded and globally Lipschitz.  More precisely, there are constants $M,L>0$ such that
\begin{equation}
    \norm{F(u)}_{\Hcal}\leq M,
    \qquad
    \norm{F(u)-F(w)}_{\Hcal}\leq L\norm{u-w}_{\Hcal}
    \label{eq:F-bounded-Lipschitz}
\end{equation}
for all $u,w\in\Hcal$.
\end{assumption}

\begin{remark}
\label{rem:preparation}
Polynomial reaction terms are generally not globally Lipschitz as maps from $L^2\to L^2$.  
In applications one can either find a priori estimates that show that solutions remain bounded, or prove dissipativity and the existence of an absorbing set in a sufficiently regular phase space, and then modify the nonlinearity outside a slightly larger bounded set so that it becomes globally bounded and globally Lipschitz.
All statements below can then be applied to the prepared equation and, once trajectories have entered the region on which the two vector fields agree, to the original dynamics.
\end{remark}

\begin{assumption}[Spectral cutoff and renormalized operator]
\label{ass:spectral_shift}
Fix $\mu>0$.  For $0<\eps\leq1$, set
\begin{equation}
    \Lambda_\eps:=\mu\eps^{-2},
    \label{eq:lambda_eps}
\end{equation}
and define the orthogonal spectral projections
\begin{equation}
    P_\eps:=\mathbf 1_{[0,\Lambda_\eps]}(A),
    \qquad
    Q_\eps:=I-P_\eps.
    \label{eq:spectral-projections}
\end{equation}  
The space $P_\eps\Hcal$ is finite-dimensional, and $Q_\eps\Hcal$ contains the modes with eigenvalues strictly above $\Lambda_\eps$.
Adding damping on the high modes gives the renormalized operator
\begin{equation}
    A_\eps:=A+\eps^{-1}Q_\eps,
    \qquad \dom(A_\eps)=\dom(A).
    \label{eq:renormalized_operator}
\end{equation}
It is again nonnegative and self-adjoint, with compact resolvent.
We can bound the spectra $\sigma(\cdot)$ under projections and write
\begin{equation}
\begin{aligned}
    \beta_\eps
        :=\sup\sigma(A|_{P_\eps\Hcal}),\quad
    \lambda_\eps^+
        :=\inf\sigma(A|_{Q_\eps\Hcal}),\quad
    \alpha_\eps
        :=\inf\sigma(A_\eps|_{Q_\eps\Hcal})
          =\lambda_\eps^++\eps^{-1},
\end{aligned}
    \label{eq:alpha_beta}
\end{equation}
with the convention $\beta_\eps=0$ if $P_\eps\Hcal=\{0\}$.
For every $0<\eps\leq1$,
\begin{equation}
    \beta_\eps\leq\Lambda_\eps=\mu\eps^{-2},
    \qquad
    \lambda_\eps^+>\Lambda_\eps,
    \qquad
    \alpha_\eps>\mu\eps^{-2}+\eps^{-1}.
    \label{eq:beta_alpha_bounds}
\end{equation}
Observe that for all sufficiently small $\eps$, the retained space after projection is nontrivial.
\end{assumption}

\begin{lemma}
\label{lem:created_gap}
Under Assumption \ref{ass:spectral_shift}, the spectral separation of $A_\eps$ between the retained and discarded subspaces is
\begin{equation}
    \gamma_\eps
    :=\alpha_\eps-\beta_\eps
    =\bigl(\lambda_\eps^+-\beta_\eps\bigr)+\eps^{-1}
    \geq\eps^{-1}.
    \label{eq:created_gap}
\end{equation}
Hence, for every fixed $L>0$, the standard gap condition from the theory of inertial manifolds
\begin{equation}
    \gamma_\eps>8L
    \label{eq:gap_condition}
\end{equation}
holds for all sufficiently small $\eps>0$.
\end{lemma}

\begin{proof}
The shift adds $\eps^{-1}$ to every high-mode eigenvalue and leaves the low-mode eigenvalues unchanged. 
Since $\lambda_\eps^+>\beta_\eps$, the estimate follows.
\end{proof}

Having controlled the spectral properties, we should also aim to control how projected solutions converge as $\eps$ tends to zero.

\begin{lemma}
\label{lem:tail}
Let $0\leq a<r$.  Then, for every $u\in\Hcal^r$,
\begin{equation}
    \norm{Q_\eps u}_a
    \leq
    (1+\Lambda_\eps)^{-(r-a)/2}\norm{u}_r
    \leq
    C_{\mu,r,a}\,\eps^{r-a}\norm{u}_r,
    \qquad 0<\eps\leq1.
    \label{eq:tail-estimate}
\end{equation}
In particular, $P_\eps u\to u$ in $\Hcal$ for every $u\in\Hcal$, and the convergence $Q_\eps u\to0$ is uniform on bounded subsets of $\Hcal^r$ in the $\Hcal^a$ norm.
\end{lemma}

\begin{proof}
Let $(\phi_j)$ be an orthonormal eigenbasis of $A$ and write $u=\sum_j u_j\phi_j$.  Since every eigenvalue represented in $Q_\eps\Hcal$ is strictly larger than $\Lambda_\eps$,
\[
    \norm{Q_\eps u}_a^2
    =\sum_{\lambda_j>\Lambda_\eps}(1+\lambda_j)^a\abs{u_j}^2\leq
      (1+\Lambda_\eps)^{-(r-a)}
      \sum_{\lambda_j>\Lambda_\eps}(1+\lambda_j)^r\abs{u_j}^2\leq
      (1+\Lambda_\eps)^{-(r-a)}\norm{u}_r^2.
\]
This proves the first part of the estimate.  
Since $\Lambda_\eps=\mu\eps^{-2}$,
\[
    (1+\Lambda_\eps)^{-(r-a)/2}
    \leq \mu^{-(r-a)/2}\eps^{r-a},
\]
the second part follows.
Strong convergence on $\Hcal$ follows by approximation with finite spectral sums.
\end{proof}

\begin{remark}
Note that at the endpoint $a=r$, one still has $Q_\eps u\to0$ in $\Hcal^r$ for each fixed $u\in\Hcal^r$.
However, $\norm{Q_\eps}_{\mathcal L(\Hcal^r)}=1$, so the convergence is not uniform on its unit ball.
\end{remark}

Under Assumptions \ref{ass:operator}, \ref{ass:nonlinearity} and \ref{ass:spectral_shift} consider the original equation
\begin{equation}
    \dot v+Av=F(v),
    \qquad v(0)=v_0,
    \label{eq:original}
\end{equation}
and, for $\eps>0$, the renormalized equation
\begin{equation}
    \dot v^\eps+A_\eps v^\eps=F(v^\eps),
    \qquad v^\eps(0)=v_0.
    \label{eq:shifted}
\end{equation}
The contraction semigroups and global Lipschitz continuity of $F$ give unique global mild solutions in $\Hcal$; see, for example, \cite{amann1995linear,lunardi2012analytic}.
We denote the corresponding semiflows by
\[
    S(t)v_0:=v(t;v_0),
    \qquad
    S_\eps(t)v_0:=v^\eps(t;v_0).
\]
It is then natural to try to find estimates for combinations of the linear(ized) evolution, the full semigroup, and various projections to be able to control the dynamics.

\begin{proposition}
\label{prop:tail-dynamics}
For every $\eps>0$ and $t\geq0$,
\begin{align}
    \norm{\e^{-tA}Q_\eps}_{\mathcal L(\Hcal)}
        &\leq \e^{-\lambda_\eps^+t},
        \label{eq:original-high-decay}\\
    \norm{\e^{-tA_\eps}Q_\eps}_{\mathcal L(\Hcal)}
        &\leq \e^{-\alpha_\eps t}.
        \label{eq:high-decay}
\end{align}
The group on the finite-dimensional low-mode space satisfies, for $t\leq0$,
\begin{equation}
    \norm{\e^{-tA}P_\eps}_{\mathcal L(P_\eps\Hcal)}
        \leq \e^{\beta_\eps\abs{t}}.
    \label{eq:low-backward}
\end{equation}
Moreover, for every $v_0\in\Hcal$ and $t\geq0$,
\begin{align}
    \norm{Q_\eps S(t)v_0}_{\Hcal}
    &\leq
      \e^{-\lambda_\eps^+t}\norm{Q_\eps v_0}_{\Hcal}
      +\frac{M}{\lambda_\eps^+},
    \label{eq:original-tail-dynamic}\\
    \norm{Q_\eps S_\eps(t)v_0}_{\Hcal}
    &\leq
      \e^{-\alpha_\eps t}\norm{Q_\eps v_0}_{\Hcal}
      +\frac{M}{\alpha_\eps}.
    \label{eq:renormalised-tail-dynamic}
\end{align}
\end{proposition}

\begin{proof}
Because $P_\eps$ and $Q_\eps$ are spectral projections of the operator $A$, they commute with $A$, with $A_\eps$, and with the corresponding semigroups.  
Moreover, on $Q_\eps\Hcal$ the spectrum of $A$ is bounded below by $\lambda_\eps^+$, while the spectrum of $A_\eps$ is bounded below by $\alpha_\eps$; see also \eqref{eq:alpha_beta} for the relevant definitions.  
The spectral theorem therefore yields \eqref{eq:original-high-decay} and \eqref{eq:high-decay}.  
Similarly, if $t\leq0$, then every eigenvalue of $A|_{P_\eps\Hcal}$ is at most $\beta_\eps$, and hence
\[
    \norm{\e^{-tA}P_\eps}_{\mathcal L(P_\eps\Hcal)}
    =\sup_{\lambda\in\sigma(A|_{P_\eps\Hcal})}\e^{\abs{t}\lambda}
    \leq\e^{\beta_\eps\abs{t}},
\]
which proves \eqref{eq:low-backward}. For the original equation, the variation-of-constants formula and commutation with $Q_\eps$ yield
\[
    Q_\eps S(t)v_0
    =\e^{-tA}Q_\eps v_0
     +\int_0^t\e^{-(t-s)A}Q_\eps F(S(s)v_0)\,\mathrm ds.
\]
Using \eqref{eq:original-high-decay} and $\norm{F}_{L^\infty(\Hcal;\Hcal)}\leq M$ gives
\[
\begin{aligned}
    \norm{Q_\eps S(t)v_0}_{\Hcal}
    &\leq
      \e^{-\lambda_\eps^+t}\norm{Q_\eps v_0}_{\Hcal}
      +M\int_0^t\e^{-\lambda_\eps^+(t-s)}\,\mathrm ds =\e^{-\lambda_\eps^+t}\norm{Q_\eps v_0}_{\Hcal}
      +\frac{M}{\lambda_\eps^+}\bigl(1-\e^{-\lambda_\eps^+t}\bigr).
\end{aligned}
\]
The proof of \eqref{eq:renormalised-tail-dynamic} is identical with $A$ replaced by $A_\eps$.  
\end{proof}

Having suitable phase spaces and estimates prepared, one can next aim to apply the theory of invariant manifold to the renormalized equation.
\section{Inertial Manifold for the Renormalized Equation}
\label{sec:inertial-manifold}

Before continuing, we note that various formulations of invariant manifold theory for PDEs are available. Here we focus on the formulation via inertial manifolds as it provides a clean entry point to the theory with relatively little technical overhead.

\begin{definition}[Inertial manifold]
\label{def:inertial-manifold}
Let $T(t)$ be a semiflow on $\Hcal$. 
A finite-dimensional Lipschitz submanifold $\Mcal\subset\Hcal$ is called an \emph{inertial manifold} for $T(t)$ if it is invariant, contains the global attractor whenever the latter exists, and has the exponential tracking property, i.e., there exist $\varkappa>0$ and $C\geq1$ such that for every $u_0\in\Hcal$ there exists $\bar u_0\in\Mcal$ satisfying
\begin{equation}
    \norm{T(t)u_0-T(t)\bar u_0}_{\Hcal}
    \leq
    C\e^{-\varkappa t}
    \norm{u_0-\bar u_0}_{\Hcal},
    \qquad t\geq0.
    \label{eq:tracking-definition}
\end{equation}
\end{definition}

\begin{theorem}[Inertial manifold for the renormalized equation]
\label{thm:inertial-manifold}
Assume Assumptions \ref{ass:operator}, \ref{ass:nonlinearity} and \ref{ass:spectral_shift}, and denote the spectral gap by \(\gamma_\eps:=\alpha_\eps-\beta_\eps\).
Suppose that
\begin{equation}
    \gamma_\eps\geq 8L.
    \label{eq:inertial-gap-condition}
\end{equation}
Then, the renormalized equation \eqref{eq:shifted} possesses a finite-dimensional Lipschitz inertial manifold
\begin{equation}
    \Mcal_\eps
    =
    \set{p+h_\eps(p):p\in P_\eps\Hcal},
    \qquad
    h_\eps:P_\eps\Hcal\longrightarrow Q_\eps\Hcal,
    \label{eq:graph-definition}
\end{equation}
where the graph map satisfies
\begin{equation}
    \norm{h_\eps}_{L^\infty(P_\eps\Hcal;\Hcal)}
    \leq
    \frac{M}{\alpha_\eps}
    \leq
    \frac{M\eps^2}{\mu+\eps}
    \leq
    \frac{M}{\mu}\eps^2,
    \label{eq:graph-sup-bound}
\end{equation}
and
\begin{equation}
    \Lip(h_\eps)
    \leq
    \frac{4L}{\gamma_\eps}
    \leq 4L\eps.
    \label{eq:graph-lip-bound}
\end{equation}
Moreover,
\begin{equation}
    S_\eps(t)\Mcal_\eps=\Mcal_\eps,
    \qquad t\geq0,
    \label{eq:strict-invariance}
\end{equation}
and the dynamics on $\Mcal_\eps$ are governed by the finite-dimensional equation
\begin{equation}
    \dot p+Ap
    =
    P_\eps F\bigl(p+h_\eps(p)\bigr),
    \qquad p\in P_\eps\Hcal.
    \label{eq:reduced-equation}
\end{equation}
Finally, there exist constants $C_\eps\geq1$ and $\varkappa_\eps>0$ such that for every $u_0\in\Hcal$ there exists $\widetilde u_0\in\Mcal_\eps$ with
\begin{equation}
    \norm{S_\eps(t)u_0-S_\eps(t)\widetilde u_0}_{\Hcal}
    \leq
    C_\eps\e^{-\varkappa_\eps t}
    \norm{u_0-\widetilde u_0}_{\Hcal},
    \qquad t\geq0.
    \label{eq:renormalized-tracking}
\end{equation}
Consequently, whenever $S_\eps(t)$ possesses a global attractor $\Acal_\eps$, one has
\[
    \Acal_\eps\subset\Mcal_\eps.
\]
\end{theorem}

\begin{proof}
    Since the renormalized system possesses a spectral gap of size $\gamma_\varepsilon\geq 8L$, the conclusions of the result follow from classical inertial manifold theory: see, for example, \cite[Chapter~VIII, Theorem~3.1]{Temam1997} and \cite[Theorem 2.1]{Zelik2014}.
\end{proof}


\section{Approximate semi-conjugacy}
\label{sec:approximate-semiconjugacy}

The inertial manifold $\Mcal_{\eps}$ provides an exact finite-dimensional reduction of the \emph{renormalized} equation.  
The question we answer in this section is \emph{how the two dynamical systems are related.}\\
To do this, we compare the reduced dynamics with the original semiflow $S(t)$ through a surjective graph retraction $R_\eps: \Hcal\to \Mcal_\eps$.
We cannot, in general, expect a bijective mapping between the phase space and the inertial manifold and as a result cannot expect the reduced system be topological conjugate to the original system.
Therefore, we consider a semi-conjugacy between the two dynamical systems, via the following relation.
\[R_\eps S(t)=S_\eps^{\Mcal}(t)R_\eps.\]
Here we estimate the defect in this identity on finite intervals and, under additional stability, uniformly in time.

Throughout this section, Assumptions~\ref{ass:operator}, \ref{ass:nonlinearity} and~\ref{ass:spectral_shift} hold, and $0<\eps\leq1$ is chosen so that the spectral gap assumption is satisfied, i.e., $\gamma_\eps\geq8L$.
We write $S_\eps^{\Mcal}(t):=S_\eps(t)|_{\Mcal_\eps}$.

\subsection{Retraction onto the inertial graph}
\label{sec:retraction}

We define the graph embedding
\begin{equation}
    J_\eps:P_\eps\Hcal\to\Mcal_\eps,
    \qquad J_\eps(p):=p+h_\eps(p),
    \label{eq:graph-embedding}
\end{equation}
and consequently the retraction is given by the composition of the low mode projection with the graph embedding
\begin{equation}
    R_\eps:=J_\eps P_\eps,
    \qquad R_\eps u=P_\eps u+h_\eps(P_\eps u).
    \label{eq:graph-retraction}
\end{equation}
Thus $R_\eps$ preserves the low modes and replaces the high component by its value on the invariant graph.

\begin{proposition}
\label{prop:retraction}
The map $R_\eps:\Hcal\to\Mcal_\eps$ has the following properties.
\begin{enumerate}[label=\textup{(\alph*)},leftmargin=2em]
\item
\begin{equation}
    R_\eps(\Hcal)=\Mcal_\eps,
    \qquad R_\eps|_{\Mcal_\eps}=\Id,
    \qquad R_\eps^2=R_\eps.
    \label{eq:retraction-idempotent}
\end{equation}
\item The low modes are preserved:
\begin{equation}
    P_\eps R_\eps u=P_\eps u,
    \qquad u\in\Hcal.
    \label{eq:retraction-low-mode}
\end{equation}
\item The retraction has the same fibres as $P_\eps$:
\begin{equation}
    R_\eps u=R_\eps w
    \quad\Longleftrightarrow\quad P_\eps u=P_\eps w.
    \label{eq:retraction-fibres}
\end{equation}
\item The embedding and retraction are globally Lipschitz, with
\begin{equation}
    \Lip(R_\eps)\leq\Lip(J_\eps)
    \leq c_\eps:=\sqrt{1+\frac{16L^2}{\gamma_\eps^2}}
    \leq1+8L^2\eps^2.
    \label{eq:retraction-lip}
\end{equation}
\item For every $u\in\Hcal$,
\begin{equation}
    \norm{R_\eps u-u}_{\Hcal}
    \leq\norm{Q_\eps u}_{\Hcal}+\frac{M}{\alpha_\eps}.
    \label{eq:retraction-error}
\end{equation}
\end{enumerate}
\end{proposition}

\begin{proof}
The identity $P_\eps J_\eps=\Id$ gives \textup{(a)}--\textup{(c)}.
By \eqref{eq:graph-lip-bound}, orthogonality yields, for $p_1,p_2\in P_\eps\Hcal$,
\[
    \norm{J_\eps p_1-J_\eps p_2}_{\Hcal}^2
    =\norm{p_1-p_2}_{\Hcal}^2
      +\norm{h_\eps(p_1)-h_\eps(p_2)}_{\Hcal}^2
    \leq c_\eps^2\norm{p_1-p_2}_{\Hcal}^2.
\]
Since $R_\eps=J_\eps P_\eps$ and $\sqrt{1+x}\leq1+x/2$ for $x\geq0$, this proves \textup{(d)}.
Finally, $R_\eps u-u=h_\eps(P_\eps u)-Q_\eps u$, so \textup{(e)} follows from \eqref{eq:graph-sup-bound}.
\end{proof}

\begin{corollary}
\label{cor:retraction-convergence}
For every $u\in\Hcal$,
\begin{equation}
    R_\eps u\longrightarrow u\quad\text{in }\Hcal
    \qquad\text{as }\eps\to0,
    \label{eq:retraction-strong}
\end{equation}
uniformly on compact subsets of $\Hcal$.
If $B\subset\Hcal^r$ is bounded for some $r>0$, then
\begin{equation}
    \sup_{u\in B}\norm{R_\eps u-u}_{\Hcal}
    \leq(1+\Lambda_\eps)^{-r/2}\sup_{u\in B}\norm{u}_r
       +\frac{M}{\alpha_\eps}
    \leq C_B\eps^r+\frac{M\eps^2}{\mu+\eps}.
    \label{eq:retraction-regular-rate}
\end{equation}
\end{corollary}

\begin{proof}
Use \eqref{eq:retraction-error} and Lemma~\ref{lem:tail}.
The uniform bound $\norm{Q_\eps}_{\mathcal L(\Hcal)}\leq1$ upgrades strong convergence to uniform convergence on compact sets.
\end{proof}

\begin{remark}
We observe that the convergence is not uniform on any ball of positive radius in $\Hcal$.
Indeed, in a ball centered at $u_c$, choose $s>0$ smaller than its radius and a unit vector $e_\eps\in Q_\eps\Hcal$.
Then, the points $u_\eps^\pm=u_c\pm s e_\eps$ have the same retraction, whereas their distance is $2s$.
Hence
\[
    \max_{\pm}\norm{R_\eps u_\eps^\pm-u_\eps^\pm}_{\Hcal}\geq s.
\]
\end{remark}

\subsection{Finite-time approximate semi-conjugacy}
\label{sec:finite-time}

Let $\Phi_\eps(t)$ be the flow of the reduced equation \eqref{eq:reduced-equation} on $P_\eps\Hcal$.
Using the strict invariance of the inertial manifold we obtain the following conjugacy between the semiflow of the renormalized system restricted onto the inertial manifold and the flow of the reduced system on the inertial manifold
\[
    S_\eps^{\Mcal}(t)J_\eps=J_\eps\Phi_\eps(t),
    \qquad t\geq0.
\]
Consequently, with $w(t;u):=P_\eps S(t)u-\Phi_\eps(t)P_\eps u$, Proposition~\ref{prop:retraction} gives
\[
    \norm{w(t;u)}_{\Hcal}
    \leq\norm{R_\eps S(t)u-S_\eps^{\Mcal}(t)R_\eps u}_{\Hcal}
    \leq c_\eps\norm{w(t;u)}_{\Hcal}.
\]
Thus the semi-conjugacy defect measures the low-mode error up to a factor $c_\eps=1+O(\eps^2)$.
This motivate the following definition.
\begin{definition}[Finite-time approximate semi-conjugacy]
\label{def:finite-time-semiconjugacy}
For $T>0$ and $u\in\Hcal$, define
\begin{equation}
    \Dcal_{\eps,T}(u)
    :=\sup_{0\leq t\leq T}
       \norm{R_\eps S(t)u-S_\eps^{\Mcal}(t)R_\eps u}_{\Hcal}.
    \label{eq:finite-time-defect}
\end{equation}
The family $(R_\eps,S_\eps^{\Mcal})$ gives a \emph{finite-time approximate semi-conjugacy} of $S(t)$ on $B\subset\Hcal$ if, for every $T>0$,
\begin{equation}
    \sup_{u\in B}\Dcal_{\eps,T}(u)\longrightarrow0
    \qquad\text{as }\eps\to0.
    \label{eq:approx-semiconjugacy-definition}
\end{equation}
For fixed $\eps$, the vanishing of the defect for all $u\in\Hcal$ and all $T>0$ means that $R_\eps$ is an exact semi-conjugacy.
\end{definition}
We refer to \cite{robinson1999dynamical} for more details on topological conjugacy and semi-conjugacy.

For brevity, we set
\begin{equation}
    d_\eps:=\frac{M}{\lambda_\eps^+}+\frac{M}{\alpha_\eps}\leq\frac{2M\eps^2}{\mu},
    \qquad
    \rho_\eps(u):=\norm{Q_\eps u}_{\Hcal}+d_\eps.
    \label{eq:rho}
\end{equation}

\begin{theorem}[Finite-time approximate semi-conjugacy]
\label{thm:finite-semiconjugacy}
Under the standing assumptions, for every $T>0$ and $u_0\in\Hcal$,
\begin{equation}
    \Dcal_{\eps,T}(u_0)
    \leq c_\eps\left[
       \frac{L(\e^{LT}-\e^{-\lambda_\eps^+T})}{L+\lambda_\eps^+}
         \norm{Q_\eps u_0}_{\Hcal}
       +(\e^{LT}-1)d_\eps
    \right].
    \label{eq:finite-semiconjugacy-estimate}
\end{equation}
Consequently, for every bounded $B\subset\Hcal$,
\begin{equation}
    \sup_{u_0\in B}\Dcal_{\eps,T}(u_0)\leq C_{B,T}\eps^2.
    \label{eq:bounded-semiconjugacy-rate}
\end{equation}
\end{theorem}

\begin{proof}
Write $p(t)=P_\eps S(t)u_0$, $q(t)=Q_\eps S(t)u_0$, $p_\eps(t)=\Phi_\eps(t)P_\eps u_0$, and $w=p-p_\eps$.
Then $w(0)=0$ and
\begin{equation}
    \dot w+Aw
    =P_\eps\bigl[F(p+q)-F(p_\eps+h_\eps(p_\eps))\bigr].
    \label{eq:low-mode-difference}
\end{equation}
Non-negativity of the operator $A$ and Lipschitz continuity of $F$ give
\begin{equation}
    \frac{d}{dt}\norm{w}_{\Hcal}
    \leq L\norm{w}_{\Hcal}
       +L\bigl(\norm{q}_{\Hcal}+\norm{h_\eps(p_\eps)}_{\Hcal}\bigr).
    \label{eq:low-mode-difference-inequality}
\end{equation}
Proposition~\ref{prop:tail-dynamics} and Theorem~\ref{thm:inertial-manifold} yield
\begin{align}
    \norm{q(t)}_{\Hcal}
    &\leq\e^{-\lambda_\eps^+t}\norm{Q_\eps u_0}_{\Hcal}
        +\frac{M}{\lambda_\eps^+},
    \label{eq:original-high-mode-bound}\\
    \norm{h_\eps(p_\eps(t))}_{\Hcal}
    &\leq\frac{M}{\alpha_\eps}.
    \label{eq:graph-bound-finite-time}
\end{align}
Scalar comparison therefore gives
\begin{align}
    \norm{w(t)}_{\Hcal}
    &\leq L\int_0^t\e^{L(t-s)}
       \bigl(\e^{-\lambda_\eps^+s}\norm{Q_\eps u_0}_{\Hcal}+d_\eps\bigr)\,\mathrm ds
    =\frac{L(\e^{Lt}-\e^{-\lambda_\eps^+t})}{L+\lambda_\eps^+}
        \norm{Q_\eps u_0}_{\Hcal}
       +(\e^{Lt}-1)d_\eps.
    \label{eq:w-estimate}
\end{align}
The distance estimate above and monotonicity of this bound in time $t$ prove \eqref{eq:finite-semiconjugacy-estimate}.
Finally, using $(L+\lambda_\eps^+)^{-1}\leq\eps^2/\mu$, $d_\eps=O(\eps^2)$, and $c_\eps\leq3/2$ give \eqref{eq:bounded-semiconjugacy-rate}.
\end{proof}

\begin{remark}
\label{rem:finite-time-character}
Note, that the factor $\e^{LT}$ allows amplification in the retained dynamics.
Uniform comparison for arbitrarily large times requires additional stability to control the persistent forcing of size $O(\eps^2)$.
\end{remark}

\subsection{Post-initial-layer shadowing}
\label{sec:post-layer}

The preceding theorem compares the reduced trajectory with $R_\eps S(t)u_0$.
However, a direct comparison with the original flow $S(t)u_0$ also requires a small reconstruction error, defined in \eqref{eq:post-layer-decomposition}, which can be of order one initially on bounded subsets of $\Hcal$.
This can be overcome by the high-mode decay, due to the parabolic smoothing, which makes this error small after an initial layer.
For $\kappa>0$, set
\begin{equation}
    \tau_\eps^\kappa:=\frac{\kappa}{\mu}\eps^2\abs{\log\eps},
    \qquad 0<\eps<1,
    \label{eq:layer-time}
\end{equation}
so that $\e^{-\Lambda_\eps\tau_\eps^\kappa}=\eps^\kappa$.
For a bounded set $B\subset\Hcal$, write
\begin{equation}
    R_B:=\sup_{u\in B}\norm{u}_{\Hcal}.
    \label{eq:RB}
\end{equation}
Estimate~\eqref{eq:original-high-mode-bound} then yields
\begin{equation}
    \sup_{u_0\in B}\sup_{t\geq\tau_\eps^\kappa}
       \norm{Q_\eps S(t)u_0}_{\Hcal}
    \leq R_B\eps^\kappa+\frac{M}{\lambda_\eps^+}.
    \label{eq:post-layer-tail}
\end{equation}
Put $\tau=\tau_\eps^\kappa$ and $u_\tau=S(\tau)u_0$.
For $t\geq\tau$, the semigroup property gives us the decomposition
\begin{equation}
\begin{aligned}
    S_\eps^{\Mcal}(t-\tau)R_\eps u_\tau-S(t)u_0
    ={}&\bigl[S_\eps^{\Mcal}(t-\tau)R_\eps u_\tau-R_\eps S(t)u_0\bigr]  +\bigl[R_\eps S(t)u_0-S(t)u_0\bigr].
    \label{eq:post-layer-decomposition}
\end{aligned}
\end{equation}
We identify the two terms as the semi-conjugacy defect after restart at time $\tau_\eps^\kappa$ and the reconstruction error. 
This motivates the following definition.

\begin{definition}[Post-initial-layer shadowing defect]
\label{def:post-layer}
For $T>\tau_\eps^\kappa$, we define the post-initial-layer shadowing defect by
\begin{equation}
    \mathfrak{D}_{\eps,T}^{\kappa}(u_0)
    :=\sup_{\tau_\eps^\kappa\leq t\leq T}
       \norm{S_\eps^{\Mcal}(t-\tau_\eps^\kappa)
          R_\eps S(\tau_\eps^\kappa)u_0-S(t)u_0}_{\Hcal}.
    \label{eq:post-layer-defect}
\end{equation}
\end{definition}

\begin{theorem}[Post-initial-layer shadowing]
\label{thm:post-layer}
Under the standing assumptions, let $B\subset\Hcal$ be bounded, $T>0$, and $\kappa>0$.
For all sufficiently small $\eps$ such that $\tau_\eps^\kappa<T$, we have that 
\begin{equation}
    \sup_{u_0\in B} \mathfrak{D}_{\eps,T}^{\kappa}(u_0)
    \leq C_T\bigl(R_B\eps^\kappa+d_\eps\bigr),
    \label{eq:post-layer-estimate}
\end{equation}
where $C_T$ is independent of $\eps$ and $B$.
In particular, the bound is $O(\eps^2)$ for $\kappa\geq2$.
\end{theorem}

\begin{proof}
Write $\tau=\tau_\eps^\kappa$ and $u_\tau=S(\tau)u_0$.
Using $\e^{-\lambda_\eps^+s}\leq1$ in \eqref{eq:w-estimate}, applied from initial time $u_\tau$, gives for $\tau\leq t\leq T$
\begin{equation}
    \norm{S_\eps^{\Mcal}(t-\tau)R_\eps u_\tau-R_\eps S(t)u_0}_{\Hcal}
    \leq c_\eps(\e^{LT}-1)\rho_\eps(u_\tau),
    \label{eq:post-layer-first-part}
\end{equation}
where the term $\rho_\eps(u_\tau)$ can be further estimated as
\begin{equation}
    \rho_\eps(u_\tau)
    \leq R_B\eps^\kappa+\frac{2M}{\lambda_\eps^+}+\frac{M}{\alpha_\eps}
    \leq2\bigl(R_B\eps^\kappa+d_\eps\bigr).
    \label{eq:post-layer-rho}
\end{equation}
For the reconstruction error we obtain, using \eqref{eq:retraction-error}, uniformly for $t\geq\tau$,
\begin{equation}
    \norm{R_\eps S(t)u_0-S(t)u_0}_{\Hcal}
    \leq R_B\eps^\kappa+d_\eps.
    \label{eq:post-layer-retraction}
\end{equation}
Combining these bounds in \eqref{eq:post-layer-decomposition} proves the claim with $C_T=1+3(\e^{LT}-1)$, since $c_\eps\leq3/2$.
\end{proof}

\subsection{Uniform-in-time approximate semi-conjugacy}
\label{sec:uniform-semiconjugacy}

In the remainder of this section we present two structural assumptions, that help to control the $e^{LT}$-term of the previous results and thus allow for uniform in-time statements.
We first impose a global contraction condition, and then, a weaker one, where we require contraction only on a reduced region containing the trajectories after a common entrance time.
Both conditions control the accumulation of the small forcing in the low-mode comparison.

\begin{assumption}[Uniform incremental dissipativity]
\label{ass:incremental-dissipativity}
There are $\eps_0\in(0,1]$ and $\sigma>0$ such that, for every $0<\eps\leq\eps_0$, $p_1,p_2\in P_\eps\Hcal$ and $q\in Q_\eps\Hcal$,
\begin{equation}
    \left\langle A(p_1-p_2)
       -P_\eps\bigl(F(p_1+q)-F(p_2+q)\bigr),p_1-p_2
    \right\rangle_{\Hcal}
    \geq\sigma\norm{p_1-p_2}_{\Hcal}^2.
    \label{eq:incremental-dissipativity}
\end{equation}
\end{assumption}

\begin{remark}
The above assumption can be understood as follows.
For each frozen high component $q$, this condition makes the retained vector field contractive with a rate independent of the cutoff.
\end{remark}

\begin{definition}[Uniform-in-time approximate semi-conjugacy]
\label{def:uniform-semiconjugacy}
For $u\in\Hcal$, define
\begin{equation}
    \Dcal_{\eps,\infty}(u)
    :=\sup_{t\geq0}\norm{R_\eps S(t)u-S_\eps^{\Mcal}(t)R_\eps u}_{\Hcal}.
    \label{eq:uniform-semiconjugacy-defect}
\end{equation}
The renormalized inertial dynamics are \emph{uniformly-in-time approximately semi-conjugate} to $S(t)$ on $B\subset\Hcal$ if
\begin{equation}
    \sup_{u\in B}\Dcal_{\eps,\infty}(u)\longrightarrow0
    \qquad\text{as }\eps\to0.
    \label{eq:uniform-semiconjugacy-definition}
\end{equation}
\end{definition}

\begin{theorem}[Uniform-in-time approximate semi-conjugacy]
\label{thm:uniform-semiconjugacy}
In addition to the standing assumptions, suppose that Assumption~\ref{ass:incremental-dissipativity} holds and $0<\eps\leq\eps_0$.
Then, for every $u_0\in\Hcal$,
\begin{equation}
    \Dcal_{\eps,\infty}(u_0)
    \leq c_\eps\left[
       \frac{L}{\lambda_\eps^+}\norm{Q_\eps u_0}_{\Hcal}
       +\frac{L}{\sigma}d_\eps
    \right].
    \label{eq:uniform-semiconjugacy-estimate}
\end{equation}
Consequently, for every bounded $B\subset\Hcal$,
\begin{equation}
    \sup_{u_0\in B}\Dcal_{\eps,\infty}(u_0)\leq C_B\eps^2.
    \label{eq:uniform-bounded-rate}
\end{equation}
\end{theorem}

\begin{proof}
We use the same notation and variables $p,q,p_\eps,w$ from the proof of Theorem~\ref{thm:finite-semiconjugacy}.
The idea now is to split the nonlinear difference in \eqref{eq:low-mode-difference} at $F(p_\eps+q)$.
By Assumption~\ref{ass:incremental-dissipativity} we can control the first difference and Lipschitz continuity controls the second, giving us
\begin{equation}
    \frac{d}{dt}\norm{w}_{\Hcal}
    \leq-\sigma\norm{w}_{\Hcal}
       +L\bigl(\norm{q}_{\Hcal}+\norm{h_\eps(p_\eps)}_{\Hcal}\bigr).
    \label{eq:uniform-scalar-ineq}
\end{equation}
Since $w(0)=0$, estimates \eqref{eq:original-high-mode-bound}--\eqref{eq:graph-bound-finite-time} imply that
\[
    \norm{w(t)}_{\Hcal}
    \leq L\int_0^t\e^{-\sigma(t-s)}
        \bigl(\e^{-\lambda_\eps^+s}\norm{Q_\eps u_0}_{\Hcal}+d_\eps\bigr)\,\mathrm ds.
\]
The two kernels can be estimated by $1/\lambda_\eps^+$ and $1/\sigma$, respectively.
Thus
\begin{equation}
    \norm{w(t)}_{\Hcal}
    \leq\frac{L}{\lambda_\eps^+}\norm{Q_\eps u_0}_{\Hcal}
        +\frac{L}{\sigma}d_\eps.
    \label{eq:uniform-w-estimate}
\end{equation}
Multiplying this estimate by $c_\eps$ and using the spectral bounds we obtain the two assertions.
\end{proof}

\begin{remark}
The global hypothesis has strong consequences.
In \eqref{eq:incremental-dissipativity}, we set $p_1=P_\eps u$, $p_2=P_\eps v$, $q=0$, and let $\eps\to0$ to obtain
\[
    \langle A(u-v)-F(u)+F(v),u-v\rangle_{\Hcal}
    \geq\sigma\norm{u-v}_{\Hcal}^2,
    \qquad u,v\in\dom(A).
\]
Hence both $S(t)$ and $S_\eps(t)$ contract distances by $\e^{-\sigma t}$.
Their time-one maps have unique fixed points, which are equilibria by commutation with the semiflows.
Their global attractors are therefore singletons.
Moreover, boundedness of $F$ forces $\lambda_1\geq\sigma$.
To see this, we test \eqref{eq:incremental-dissipativity} with $p_2=q=0$ and let the amplitude of $p_1$ along a first eigenvector tend to infinity.
Thus this global condition excludes a nontrivial kernel of $A$.
\end{remark}

\begin{remark}
\label{rem:sufficient-incremental-dissipativity}
Two useful sufficient conditions for Assumption~\ref{ass:incremental-dissipativity} are:
\begin{enumerate}[label=\textup{(\roman*)},leftmargin=2.2em]
\item \emph{Differential criterion.}
If $F\in C^1(\Hcal;\Hcal)$ and
\begin{equation}
    \langle A\xi-P_\eps DF(u)\xi,\xi\rangle_{\Hcal}
    \geq\sigma\norm{\xi}_{\Hcal}^2
    \label{eq:differential-contraction}
\end{equation}
for every $u\in\Hcal$, $\xi\in P_\eps\Hcal$ and all sufficiently small $\eps$, with $\sigma>0$ independent of $\eps$, integrate along the segment from $p_2+q$ to $p_1+q$.
\item \emph{Coercivity and a one-sided Lipschitz bound.}
Suppose, uniformly in $\eps$,
\begin{equation}
    \langle A\xi,\xi\rangle_{\Hcal}
    \geq\nu\norm{\xi}_{\Hcal}^2,
    \qquad \xi\in P_\eps\Hcal,
    \label{eq:A-coercive-retained}
\end{equation}
and
\begin{equation}
    \left\langle P_\eps\bigl(F(p_1+q)-F(p_2+q)\bigr),p_1-p_2\right\rangle_{\Hcal}
    \leq\ell\norm{p_1-p_2}_{\Hcal}^2
    \label{eq:F-one-sided-retained}
\end{equation}
for every $p_1,p_2\in P_\eps\Hcal$, $q\in Q_\eps\Hcal$, with $\nu>\ell$ and $\nu>0$.
Then one may take $\sigma=\nu-\ell$; in particular, $\lambda_1>L$ suffices.
\end{enumerate}
\end{remark}

Next, we want to weaken the above assumption to allow for richer dynamics.
The idea is apply Assumption \ref{ass:incremental-dissipativity} locally near an attracting state.
Therefore, it suffices to control the reduced vector field only on the region occupied by the two low-mode trajectories after a common entrance time, which we denote by $t_*$.

\begin{assumption}[Incremental stability on a reduced region]
\label{ass:local-reduced-contraction}
Let $V_\eps\subset P_\eps\Hcal$.
There is $\sigma>0$, independent of sufficiently small $\eps$, such that
\begin{equation}
    \left\langle A(p_1-p_2)
      -P_\eps\bigl(F(p_1+h_\eps(p_1))-F(p_2+h_\eps(p_2))\bigr),p_1-p_2
    \right\rangle_{\Hcal}
    \geq\sigma\norm{p_1-p_2}_{\Hcal}^2
    \label{eq:local-reduced-contraction}
\end{equation}
for all $p_1,p_2\in V_\eps$.
\end{assumption}

\begin{theorem}[Uniform comparison after entrance into a stable region]
\label{thm:stable-region-semiconjugacy}
Under the standing assumptions, let $B\subset\Hcal$ be bounded.
Suppose that Assumption~\ref{ass:local-reduced-contraction} holds and there is $t_*\geq0$, independent of sufficiently small $\eps$, such that
\[
    P_\eps S(t)u_0\in V_\eps,
    \qquad \Phi_\eps(t)P_\eps u_0\in V_\eps,
    \qquad u_0\in B,\quad \text{for all}~ t\geq t_*.
\]
Then
\begin{equation}
    \sup_{u_0\in B}\Dcal_{\eps,\infty}(u_0)\leq C_B\eps^2.
    \label{eq:stable-region-uniform-result}
\end{equation}
\end{theorem}

\begin{proof}
We use again the notation of Theorem~\ref{thm:finite-semiconjugacy}, which gives $\norm{w(t_*)}_{\Hcal}\leq C_{B,t_*}\eps^2$ uniformly on $B$ (and $w(0)=0$ if $t_*=0$).
For $t\geq t_*$, we split the nonlinear difference at $F(p+h_\eps(p))$.
From Assumption~\ref{ass:local-reduced-contraction} we thgen obtain
\[
    \frac{d}{dt}\norm{w}_{\Hcal}
    \leq-\sigma\norm{w}_{\Hcal}
       +L\bigl(\norm{q}_{\Hcal}+\norm{h_\eps(p)}_{\Hcal}\bigr).
\]
Using the high-mode and graph bounds as in Theorem~\ref{thm:uniform-semiconjugacy}, but integrating now from $t_*$, yields
\[
    \norm{w(t)}_{\Hcal}
    \leq\e^{-\sigma(t-t_*)}\norm{w(t_*)}_{\Hcal}
       +\frac{L\e^{-\lambda_\eps^+t_*}}{\lambda_\eps^+}
          \norm{Q_\eps u_0}_{\Hcal}
       +\frac{L}{\sigma}d_\eps
    \leq C_B\eps^2.
\]
The bound for $J_\eps$ proves the uniform estimate for $t\geq t_*$, whereas Theorem~\ref{thm:finite-semiconjugacy} controls the interval $[0,t_*]$.
\end{proof}

The above result has a particularly interesting application.
\begin{remark}
\label{rmk:basin-of-attraction}
Let $u_*$ be an asymptotically stable equilibrium of the original equation, with basin of attraction
\[
    \Bcal(u_*):=\{u\in\Hcal:S(t)u\to u_*\text{ as }t\to\infty\},
\]
and let $K\subset\Bcal(u_*)$ be compact.
If the sets $V_\eps$, satisfying Assumption~\ref{ass:local-reduced-contraction}, contain both $P_\eps S(t)K$ and $\Phi_\eps(t)P_\eps K$ for every $t\geq t_K$, with $t_K$ independent of sufficiently small $\eps$, then
\begin{equation}
    \sup_{u_0\in K}\sup_{t\geq0}
       \norm{R_\eps S(t)u_0-S_\eps^{\Mcal}(t)R_\eps u_0}_{\Hcal}
    \leq C_K\eps^2.
    \label{eq:basinwise-uniform-semiconjugacy}
\end{equation}
Hence, we can apply Theorem~\ref{thm:stable-region-semiconjugacy} with $B=K$ and $t_*=t_K$.
\end{remark}


\section{Three examples for inertial-manifold reduction}
\label{sec:three-benchmarks}

In the previous sections we have introduced the spectral renormalization, showed the existence of an inertial manifold for the renormalized equations and compared its dynamics with the original ones, via their respective semiflows.
The question we answer in this section is rather, \emph{how do our results compare with known existence and non-existence results for inertial manifolds.}

The artificial spectral gap construction is useful only if it behaves consistently in regimes where the original equation is already reducible, and remains still meaningful when such a reduction fails.
We therefore consider three complementary examples.  
On a bounded one-dimensional interval the Laplacian itself provides gaps of precisely the scale introduced by our regularization.  
On the three-dimensional periodic cube the classical spectral-gap condition fails in general, but spatial averaging still yields an inertial manifold for scalar reaction-diffusion equations.  
Finally, Zelik's counterexample shows that, without additional structure, an inertial manifold may genuinely not exist.  

\subsection{The one-dimensional interval}
\label{subsec:one-dimensional-benchmark}

Let \(\Omega=(0,\pi)\), \(\Hcal=L^2(\Omega)\), and
\[
    A=-d\partial_x^2,
    \qquad d>0,
\]
with homogeneous Neumann or Dirichlet boundary conditions.
The distinct eigenvalues are 
\[\lambda_n=d n^2,\quad \text{with}\quad \lambda_{n+1}-\lambda_n= d(2n+1),\]
where \(n\in\mathbb N_0\) in the Neumann case and \(n\in\mathbb N\) in the Dirichlet case.

\begin{remark}
Note that for the one-dimensional Laplacian, eigenvalues of order $O(n^2)$ have consecutive spectral gaps of order $O(n)$. 
Hence, identifying the cutoff scale $n^2\sim \eps^{-2}$ gives a natural gap scale $n\sim \eps^{-1}$. 
This is precisely the scaling used for the artificial spectral gap introduced in Section \ref{sec:functional-setting}, and provides our primary motivation for that choice.
\end{remark}

Now, we use the spectral cutoff from Section \ref{sec:functional-setting}, setting $\Lambda_\eps=\mu\eps^{-2}$, and letting \(N_\eps\) be the largest index with \(\lambda_{N_\eps} \leq\Lambda_\eps\).
Then,
\begin{equation}
    N_\eps=\sqrt{\frac{\mu}{d}}\,\eps^{-1}+O(1),
    \qquad
    \lambda_{N_\eps+1}-\lambda_{N_\eps}
    =2\sqrt{d\mu}\,\eps^{-1}+O(1).
    \label{eq:1d-natural-gap}
\end{equation}
Thus the natural gap at eigenvalues of size \(O(\eps^{-2})\) is already of size \(O(\eps^{-1})\), exactly the scaling of the artificial jump.

\begin{proposition}
\label{prop:1d-two-inertial-manifolds}
Assume Assumption \ref{ass:nonlinearity}.  
For all sufficiently small \(\eps>0\), both the original equation
\[
    \dot v+Av=F(v)
\]
and the renormalized equation
\[
    \dot v^\eps+(A+\eps^{-1}Q_\eps)v^\eps=F(v^\eps)
\]
possess Lipschitz inertial manifolds over the same space \(P_\eps\Hcal\), denoted by
\[
    \Mcal_\eps^0
    =\{p+h_\eps^0(p):p\in P_\eps\Hcal\},
    \qquad
    \Mcal_\eps^1
    =\{p+h_\eps^1(p):p\in P_\eps\Hcal\}.
\]
Moreover,
\begin{equation}
    \norm{h_\eps^j}_{L^\infty(P_\eps\Hcal;\Hcal)}
    \leq C\eps^2,
    \qquad
    \Lip(h_\eps^j)\leq C\eps,
    \qquad j\in\{0,1\},
    \label{eq:1d-graph-bounds}
\end{equation}
and consequently
\begin{equation}
    \sup_{p\in P_\eps\Hcal}
    \norm{h_\eps^1(p)-h_\eps^0(p)}_{\Hcal}
    \leq C\eps^2.
    \label{eq:1d-graph-comparison}
\end{equation}
The reduced vector fields $X_\eps^j(p)=-Ap+P_\eps F(p+h_\eps^j(p))$ satisfy
\[\sup_{p\in P_\eps\Hcal}\norm{X_\eps^1(p)-X_\eps^0(p)}_{\Hcal}\leq C\eps^2.\]
\end{proposition}

\begin{proof}
For the original equation, \eqref{eq:1d-natural-gap} implies the classical spectral-gap condition for all sufficiently small \(\eps\); see, for example, \cite{FoiasSellTemam1988,ConstantinEtAl1989}.  
The standard Lyapunov-Perron estimates yield
\[
    \norm{h_\eps^0}_\infty
    \leq \frac{M}{\lambda_{N_\eps+1}}
    =O(\eps^2),
    \qquad
    \Lip(h_\eps^0)
    \leq \frac{4L}{\lambda_{N_\eps+1}-\lambda_{N_\eps}}
    =O(\eps).
\]
For the renormalized equation the interface gap is increased by \(\eps^{-1}\), while the first discarded eigenvalue remains of order \(\eps^{-2}\).  
Hence Theorem \ref{thm:inertial-manifold} yields the same orders for \(h_\eps^1\).
Estimate \eqref{eq:1d-graph-comparison} follows by the triangle inequality.  
Since the low-mode linear part is identical for both reductions,
\[
    X_\eps^j(p)=-Ap+P_\eps F\bigl(p+h_\eps^j(p)\bigr),
\]
the Lipschitz continuity of \(F\) gives
\[
    \sup_p\norm{X_\eps^1(p)-X_\eps^0(p)}
    \leq L\sup_p\norm{h_\eps^1(p)-h_\eps^0(p)}
    =O(\eps^2).
\]
\end{proof}

Thus we can conclude that in this example the artificial shift does not create a new mechanism; it only strengthens a spectral splitting already present in the underlying PDE. 

\begin{remark}
The same observation also removes the need for an artificial gap in the slow-manifold construction of Section \ref{sec:fast-slow}.  
Indeed, for the unmodified slow operator the relevant stable/slow separation is
\begin{equation}
    \min\left\{\omega_f\del^{-1},\lambda_{N_\eps+1}\right\}
    -\lambda_{N_\eps}.
    \label{eq:1d-slow-gap}
\end{equation}
If \(\del\leq\vartheta\eps^2\) with \(\vartheta<\omega_f/\mu\), then the fast relaxation rate \(\omega_f\del^{-1}\) lies above the retained slow spectrum by an amount of order \(\eps^{-2}\), while the natural separation in the slow spectrum is of order \(\eps^{-1}\) by \eqref{eq:1d-natural-gap}.  
Hence the Lyapunov-Perron gap condition used in Theorem \ref{thm:two-parameter-manifold} holds for the original fast-slow problem.  
Thus both an inertial manifold for the slow equation and the corresponding slow manifold for the fast-slow system exist intrinsically.  
\end{remark}

\subsection{The three-dimensional periodic cube}
\label{subsec:three-dimensional-spatial-averaging}

In the second example we consider a scalar reaction-diffusion equation on the three-dimensional periodic cube $\Omega=(-\pi,\pi)^3$, where we set $\Hcal=L^2_{\rm per}(\Omega)$ and $ A=I-\Delta$. 
The eigenfunctions of $A$ are the Fourier modes $e^{ik\cdot x}$, $k\in\mathbb Z^3$, with eigenvalues $\lambda_k=1+|k|^2$.

In contrast to the one-dimensional situation of the previous example, the consecutive distinct spectral levels do not develop gaps tending to infinity.  
Thus increasing the cutoff cannot guarantee a gap exceeding an arbitrary fixed nonlinear Lipschitz threshold.

Nevertheless, the scalar reaction-diffusion problem provides an important positive result beyond the classical inertial manifold theorem.
Mallet-Paret and Sell showed that, on suitable higher-dimensional domains including the three-dimensional cube, the missing adjacent spectral gap can be replaced by a \emph{spatial-averaging property}; see \cite{MalletParetSell1988}.  
A modern formulation and extensions of this construction are given in \cite{KostiankoEtAl2022}.
We briefly recall the mechanism relevant for the comparison below.

Consider the scalar equation
\begin{equation}
    \partial_tu+Au=F(u),
    \label{eq:SA-original}
\end{equation}
where $F(u)(x)=f(u(x))$ with $f$ smooth and $f'$ bounded.
Its directional derivative is
\[
    F'(u)v=f'(u(\cdot))v.
\]
The main idea of the spatial averaging is based on the fact that, on suitable bands of Fourier modes near a prescribed spectral cutoff, this multiplication operator acts almost like multiplication by a spatially constant coefficient.
More precisely, for any $\eta>0$, one can choose arbitrarily large cutoffs and corresponding band projections $\Rcal_N$ such that
\begin{equation}
    \sup_{u\in\Bcal}
    \norm{\Rcal_N F'(u)\Rcal_N-a(u)\Rcal_N}_{\mathcal L(\Hcal)}
    \leq\eta,
    \label{eq:SA-averaging-principle}
\end{equation}
where $\Bcal$ is a sufficiently regular absorbing set and $a(u)$ is the spatial mean of $f'(u)$.
The scalar term $a(u)$ shifts the growth or decay rates of all modes in the band equally, so it does not affect their relative separation.
Thus the part of the linearized reaction acting on the critical shell is asymptotically close to a scalar multiple of the identity.  
This scalar contribution can be absorbed into the linear part, and the remaining oscillatory component is sufficiently small to recover the cone and squeezing estimates required for an inertial manifold.  
In this way spatial averaging replaces a large adjacent spectral gap by additional information about the interaction between the Fourier spectrum and the multiplication operator $F'(u)$.

The preparation used in this construction matters for comparison.
Spatial averaging can produce a Lipschitz graph
\begin{equation}
    \Mcal_j^{\rm SA}
    =\{p+h_j^{\rm SA}(p):p\in P_j\Hcal\}
    \label{eq:SA-manifold}
\end{equation}
containing the original attractor $\Acal$, while global invariance of the graph holds only for a further prepared equation; the preparation may depend on $j$ \cite[Section~5]{KostiankoEtAl2022}.
It therefore gives an exact reduction of the original dynamics on $\Acal$, but does not automatically give a globally invariant graph for the same fixed $F$ in \eqref{eq:SA-original}.

We now apply the spectral renormalization approach to the above scalar equation.
We observe that both constructions use a spectral cutoff to separate retained and discarded modes.
However, spatial averaging exploits the structure of the linearized reaction near the cutoff, whereas our renormalization adds damping above it to create a spectral gap.
To compare the resulting reductions over the same retained subspace, choose $\Lambda_j\to\infty$ strictly between the last retained and first discarded spectral levels associated with $P_j$, and set
\begin{equation}
    \eps_j=\sqrt{\frac{\mu}{\Lambda_j}},
    \qquad P_{\eps_j}=P_j.
    \label{eq:SA-matched-epsilon}
\end{equation}
For a fixed $F$ satisfying Assumption~\ref{ass:nonlinearity}, Theorem~\ref{thm:inertial-manifold} then gives the exact invariant graph
\begin{equation}
    \Mcal_j^{\rm reg}
    =\{p+h_j^{\rm reg}(p):p\in P_j\Hcal\}
    \label{eq:SA-renormalized-manifold}
\end{equation}
of the modified full equation for all sufficiently large $j$.

\begin{proposition}
\label{prop:SA-calibration}
Assume Assumption~\ref{ass:nonlinearity}, and suppose \eqref{eq:SA-original} has a compact global attractor $\Acal\subset\Mcal_j^{\rm SA}$ at the interfaces above, with each $h_j^{\rm SA}$ Lipschitz.
Set $K_j=P_j\Acal$.
For all sufficiently large $j$,
\begin{equation}
    \sup_{p\in K_j}\norm{h_j^{\rm SA}(p)}_{\Hcal}
    +\norm{h_j^{\rm reg}}_\infty
    \leq\frac{2M}{\Lambda_j}
    =\frac{2M}{\mu}\eps_j^2.
    \label{eq:SA-graph-bounds}
\end{equation}
The fibre-preserving map
\begin{equation}
    \Theta_j(p+h_j^{\rm SA}(p)):=p+h_j^{\rm reg}(p)
    \label{eq:SA-fibre-map}
\end{equation}
satisfies
\begin{equation}
    \sup_{u\in\Acal}\norm{\Theta_j u-u}_{\Hcal}
    \leq C\eps_j^2.
    \label{eq:SA-manifold-distance}
\end{equation}
For $X_j^\ell(p):=-Ap+P_jF(p+h_j^\ell(p))$, $\ell\in\{{\rm SA},{\rm reg}\}$,
\begin{equation}
    \sup_{p\in K_j}\norm{X_j^{\rm SA}(p)-X_j^{\rm reg}(p)}_{\Hcal}
    \leq C\eps_j^2.
    \label{eq:SA-vector-field-distance}
\end{equation}
For $p_0\in K_j$, the original trajectory through $p_0+h_j^{\rm SA}(p_0)$ and the renormalized graph trajectory through $p_0+h_j^{\rm reg}(p_0)$ are $O(\eps_j^2)$-close on every fixed time interval, uniformly in $p_0$ and $j$.
\end{proposition}

\begin{proof}
Strict invariance and boundedness of $\Acal$, applied to the high-mode variation-of-constants formula from time $-T$ to $0$ and then letting $T\to\infty$, give
$\sup_{a\in\Acal}\norm{Q_{\eps_j}a}\leq M/\lambda_{\eps_j}^+$.
Since $\Acal$ lies on the graph, its high components are precisely $h_j^{\rm SA}(K_j)$.
Together with $\norm{h_j^{\rm reg}}_\infty\leq M/\alpha_{\eps_j}$, this proves \eqref{eq:SA-graph-bounds}.
The bounds for $\Theta_j$ and the vector fields follow from the triangle inequality and the Lipschitz bound for $F$.

The original trajectory stays in $\Acal$, so its retained coordinate $p^{\rm SA}(t)$ stays in $K_j$ and solves $\dot p^{\rm SA}=X_j^{\rm SA}(p^{\rm SA})$.
For the difference $w=p^{\rm SA}-p^{\rm reg}$, nonnegativity of $A$ gives, almost everywhere,
\[
    \frac{d}{dt}\norm{w}_{\Hcal}
    \leq L\norm{w}_{\Hcal}
       +L\left(\sup_{K_j}\norm{h_j^{\rm SA}}+\norm{h_j^{\rm reg}}_\infty\right).
\]
Since $w(0)=0$, Gronwall's inequality and \eqref{eq:SA-graph-bounds} control the retained coordinates and then the reconstructed trajectories.
\end{proof}

\begin{remark}
    Note that, if $\Mcal_j^{\rm SA}$ is globally invariant for the same fixed $F$, these estimates extend to all $p\in P_j\Hcal$ and all initial points on the graph.
\end{remark}

We conclude this example by two observations.\\
First, spatial averaging is the stronger result whenever it applies: it produces exact reduced dynamics on the original attractor.
The artificial gap should therefore not be viewed as an improvement of the spatial-averaging theorem.  
Instead, Proposition \ref{prop:SA-calibration} shows that, at increasingly fine matched interfaces, the two finite-dimensional descriptions become asymptotically indistinguishable on bounded time intervals.

Second, the two constructions use fundamentally different information.  
Spatial averaging exploits the scalar multiplication structure of $F'(u)$ and arithmetic properties of suitable spectral shells.  
The artificial construction modifies only the linear high-frequency tail and therefore does not depend on the derivative being scalar.  
For a genuinely coupled reaction-diffusion system, as explored in Section \ref{sec:coupled-AC-benchmark}, $F'(u)$ is a matrix-valued multiplication operator and its spatial average is in general a matrix rather than a scalar multiple of the identity.  
The scalar spatial-averaging mechanism does not directly close in this situation; see the discussion in \cite{KostiankoEtAl2022}.  
By contrast, the artificial spectral shift is unchanged for scalar and vector-valued equations.

\subsection{The Zelik counterexample}
\label{subsec:zelik-counterexample}

The preceding two examples concern situations in which the original equation does possess an inertial manifold, either because a large spectral gap is present intrinsically or because additional scalar structure replaces it.  
In this example we show that neither mechanism can be expected in general.  
In particular, the non-existence of an inertial manifold may reflect a genuine dynamical obstruction rather than merely the failure of a known existence criterion.

Let $A$ be positive and self-adjoint with compact inverse on a separable Hilbert space $\Hcal$, with eigenvalues
\[
    0<\lambda_1\leq\lambda_2\leq\cdots,
    \qquad
    \lambda_n\to\infty,
\]
and assume that the consecutive gaps are uniformly bounded, i.e.,
\begin{equation}
    L_0
    :=
    \sup_{n\in\mathbb N}
    (\lambda_{n+1}-\lambda_n)
    <\infty.
    \label{eq:zelik-bounded-gap-short}
\end{equation}
In this setting Zelik showed that, for every
\begin{equation}
    L>
    \max
    \left\{
       \frac{L_0}{2},
       \lambda_2
    \right\},
    \label{eq:zelik-L-condition-short}
\end{equation}
one can construct a smooth globally Lipschitz nonlinearity $F_Z$ such that
\begin{equation}
    \partial_tu+Au=F_Z(u)
    \label{eq:zelik-equation-short}
\end{equation}
has a compact global attractor $\Acal_Z$ containing two distinct complete trajectories $u_1,u_2$ for which
\begin{equation}
    \norm{u_1(t)-u_2(t)}_{\Hcal}
    \leq
    C\e^{-\kappa t^2},
    \qquad
    t\geq0;
    \label{eq:zelik-superexponential-short}
\end{equation}
see \cite[Theorem 4.4]{Zelik2014}.

\begin{remark}
Let us briefly explain why the estimate \eqref{eq:zelik-superexponential-short} is significant.
Suppose that the dynamics on $\Acal_Z$ admitted a finite-dimensional Lipschitz inertial form
\[
    \dot x=X(x)
\]
and that the identification of the attractor with the finite-dimensional invariant set were bi-Lipschitz.  
If $x_1,x_2$ are two distinct complete solutions of this ODE and $L_X$ is the Lipschitz constant of $X$, then Gronwall's inequality applied backwards along the complete trajectories gives
\begin{equation}
    |x_1(t)-x_2(t)|
    \geq
    \e^{-L_Xt}
    |x_1(0)-x_2(0)|,
    \qquad t\geq0.
    \label{eq:zelik-exponential-lower-bound}
\end{equation}
Thus two distinct complete trajectories of a finite-dimensional Lipschitz flow cannot converge faster than exponentially.
The super-exponential pair \eqref{eq:zelik-superexponential-short} violates this necessary condition.
Consequently the attractor $\Acal_Z$ is not contained in any finite-dimensional Lipschitz inertial manifold and does not admit a bi-Lipschitz finite-dimensional inertial form.
\end{remark}

The mechanism producing the super-exponential pair is intrinsically infinite-dimensional.  
To see this, one first constructs a time-periodic linear problem
\begin{equation}
    \partial_tw+Aw=\Phi(t)w
    \label{eq:zelik-Floquet}
\end{equation}
whose period map acts essentially as a weighted shift on the eigenbasis of $A$.  
During successive periods, mass is transferred to increasingly high eigenmodes, where the parabolic operator damps it more strongly.  
If $\lambda_n$ grows linearly with $n$, the multipliers accumulated after $N$ transfers behave schematically as
\[
    \prod_{n=1}^N\e^{-c\lambda_n}
    \sim
    \exp\left(
       -c\sum_{n=1}^N n
    \right)
    =
    \e^{-c'N^2}.
\]
Since $N$ is proportional to physical time, this produces the $\e^{-\kappa t^2}$ decay in \eqref{eq:zelik-superexponential-short}.  
The periodic linear system is then embedded into an autonomous dissipative semilinear equation by using additional low-dimensional dynamics to generate the
periodic coefficients; see \cite[Section~4]{Zelik2014}.  
The resulting example therefore exploits an infinite cascade through higher and higher spatial modes, precisely the type of mechanism which cannot occur in a fixed-dimensional Lipschitz ODE.

So how does our renormalization approach fit into this.
As before, we choose a smooth cutoff in $\norm{u}_{\Hcal}^2$ that equals one on a neighborhood of $\Acal_Z$ and vanishes outside a larger ball, and let $\widehat F_Z$ be the resulting preparation.
Then, there are finite constants $M,L>0$ such that
\begin{equation}
    \norm{\widehat F_Z(u)}_{\Hcal}\leq M,
    \qquad
    \norm{\widehat F_Z(u)-\widehat F_Z(v)}_{\Hcal}
    \leq L\norm{u-v}_{\Hcal}.
    \label{eq:zelik-prepared}
\end{equation}
Note, that the constant $L$ may exceed $L_Z$.
Under this construction the set $\Acal_Z$ and its super-exponential pair remain invariant, and its global attractor may be larger, but the original obstruction persists.
Next, we define $P_\eps,Q_\eps,A_\eps$ as in Assumption~\ref{ass:spectral_shift}.
Although the natural gaps are bounded, the modified spectral interface satisfies
\begin{equation}
    \alpha_\eps-\beta_\eps\geq\eps^{-1}.
    \label{eq:zelik-created-gap}
\end{equation}
In the following proposition, $S(t)$ denotes the semiflow of the prepared equation
$\partial_tu+Au=\widehat F_Z(u)$, and all graph maps use $\widehat F_Z$.

\begin{proposition}
\label{prop:regularized-zelik}
For all sufficiently small $\eps>0$, the renormalized equation
\begin{equation}
    \partial_tu^\eps+A_\eps u^\eps=\widehat F_Z(u^\eps)
    \label{eq:zelik-renormalized-equation}
\end{equation}
possesses a finite-dimensional Lipschitz inertial manifold
\begin{equation}
    \Mcal_\eps^Z=\{p+h_\eps^Z(p):p\in P_\eps\Hcal\},
    \label{eq:zelik-renormalized-IM}
\end{equation}
with
\begin{equation}
    \norm{h_\eps^Z}_\infty\leq C\eps^2,
    \qquad \Lip(h_\eps^Z)\leq C\eps.
    \label{eq:zelik-renormalized-graph-bounds}
\end{equation}
For every bounded $B\subset\Hcal$ and fixed $T>0$,
\begin{equation}
    \sup_{u_0\in B}\sup_{0\leq t\leq T}
    \norm{R_\eps S(t)u_0-S_\eps^\Mcal(t)R_\eps u_0}_{\Hcal}
    \leq C_{B,T}\eps^2.
    \label{eq:zelik-semiconjugacy-short}
\end{equation}
For fixed $\kappa\geq2$, let $\tau_\eps^\kappa=(\kappa/\mu)\eps^2|\log\eps|$.
Then, for small $\eps$ with $\tau_\eps^\kappa\leq T$,
\begin{equation}
    \sup_{u_0\in B}\sup_{\tau_\eps^\kappa\leq t\leq T}
    \norm{S(t)u_0-S_\eps^\Mcal(t-\tau_\eps^\kappa)
          R_\eps S(\tau_\eps^\kappa)u_0}_{\Hcal}
    \leq C_{B,T,\kappa}\eps^2.
    \label{eq:zelik-post-layer-short}
\end{equation}
\end{proposition}

\begin{proof}
This follows directly from Theorems \ref{thm:inertial-manifold}, \ref{thm:finite-semiconjugacy} and \ref{thm:post-layer}.
\end{proof}

This clarifies the interpretation of the artificial spectral gap and renormalization approach.
The artificial inertial manifold $\Mcal_\eps^Z$ is not a perturbation of an inertial manifold of \eqref{eq:zelik-equation-short}, since the latter does not possess one.  
Instead, it is an exact finite-dimensional invariant manifold of a nearby, resolution-dependent equation whose flow approximates the original dynamics in the one-sided sense quantified by \eqref{eq:zelik-semiconjugacy-short} and \eqref{eq:zelik-post-layer-short}.

The many-to-one nature of $R_\eps$ is therefore essential rather than technical.  
If the comparison could be upgraded uniformly to a bi-Lipschitz conjugacy between the original attractor and a finite-dimensional Lipschitz flow, the super-exponential pair
\eqref{eq:zelik-superexponential-short} would contradict \eqref{eq:zelik-exponential-lower-bound}.  
Thus the example marks a genuine boundary of what any finite-dimensional reduction with
Lipschitz inverse can preserve.

\section{Dynamical consequences of the regularization}
\label{sec:dynamical-consequences}

In the previous section we considered three applications of the spectral-gap renormalization, where we covered two settings in which an inertial manifold exists intrinsically and one in which it does not.
The next natural question to ask is therefore \emph{which dynamical properties are preserved in the regularization process.}\\
In particular, looking back at Section \ref{sec:approximate-semiconjugacy} we have three increasingly stronger comparison principles between the original semiflow $S(t)$ and the finite-dimensional inertial dynamics of the renormalized problem: finite-time approximate semi-conjugacy, post-initial-layer shadowing on bounded sets, and, under uniform incremental dissipativity, a uniform-in-time approximate semi-conjugacy.  

Throughout, $S(t)$ denotes the original semiflow, $S_\eps(t)$ the renormalized semiflow, and $S_\eps^{\Mcal}(t)$ the restriction of $S_\eps(t)$ to the inertial manifold $\Mcal_\eps$.  

\subsection{From approximate semi-conjugacy to dynamical information}
\label{sec:semiconjugacy-information}

\begin{proposition}
\label{prop:observables-semiconjugacy}
Let $B\subset\Hcal$ and set
\[
    \eta_{\eps,T}(B):=
    \sup_{u\in B}\Dcal_{\eps,T}(u)  =\sup_{u\in B}\sup_{0\le t\le T}
      \norm{R_\eps S(t)u-S_\eps^{\Mcal}(t)R_\eps u}_{\Hcal}.
\]
Then, for every Lipschitz observable $\Psi:\Mcal_\eps\to\R$ of our system,
\begin{equation}
\sup_{u\in B}\sup_{0\le t\le T}
\left|
\Psi(R_\eps S(t)u)
-
\Psi(S_\eps^{\Mcal}(t)R_\eps u)
\right|
\le \Lip(\Psi)\eta_{\eps,T}(B).
\label{eq:observable-preservation}
\end{equation}
Moreover, for every nonempty closed $K\subset\Mcal_\eps$,
\begin{equation}
\sup_{u\in B}\sup_{0\le t\le T}
\left|
\dist(R_\eps S(t)u,K)
-
\dist(S_\eps^{\Mcal}(t)R_\eps u,K)
\right|
\le \eta_{\eps,T}(B).
\label{eq:distance-observable-preservation}
\end{equation}
\end{proposition}

\begin{proof}
Fix $u\in B$ and $t\in[0,T]$. Since $\Psi$ is Lipschitz on $\Mcal_\eps$, we have
\[
\begin{aligned}
\left|
\Psi(R_\eps S(t)u)
-
\Psi(S_\eps^{\Mcal}(t)R_\eps u)
\right|
&\le
\Lip(\Psi)
\norm{
R_\eps S(t)u
-
S_\eps^{\Mcal}(t)R_\eps u
}_{\Hcal}.
\end{aligned}
\]
Taking the supremum over $t\in[0,T]$ and $u\in B$ yields \eqref{eq:observable-preservation}.

For the second assertion, recall that the distance function to a nonempty closed set $K\subset\Mcal_\eps$ is one-Lipschitz with respect to the ambient $\Hcal$-norm.
Thus we obtain
\[
\left|
\dist(R_\eps S(t)u,K)
-
\dist(S_\eps^{\Mcal}(t)R_\eps u,K)
\right|
\le
\norm{
R_\eps S(t)u
-
S_\eps^{\Mcal}(t)R_\eps u
}_{\Hcal}.
\]
Again, taking the supremum over $u\in B$ and $t\in[0,T]$ yields \eqref{eq:distance-observable-preservation}.
\end{proof}

This shows that resolved observables and distances to specified regions are controlled by the semi-conjugacy defect.
For Lipschitz observables defined on $\Hcal$, the direct trajectory estimate of Theorem~\ref{thm:post-layer} gives the corresponding comparison with $S(t)u$ after the initial layer and restart.

For nonempty sets $B,C\subset\Hcal$, we define the one-sided Hausdorff semi-distance
\begin{equation}
    \dist_{\Hcal}(B,C)
    :=
    \sup_{b\in B}\inf_{c\in C}\norm{b-c}_{\Hcal},
    \label{eq:hausdorff-semidistance}
\end{equation}
and the Hausdorff distance
\begin{equation}
    d_{H,\Hcal}(B,C)
    :=
    \max\set{\dist_{\Hcal}(B,C),\dist_{\Hcal}(C,B)}.
    \label{eq:hausdorff-distance}
\end{equation}

\begin{proposition}
\label{prop:omega-limit-semiconjugacy}
Let $u\in\Hcal$ and assume that the two forward trajectories
\[
    \{R_\eps S(t)u:t\ge0\}
    \qquad\text{and}\qquad
    \{S_\eps^{\Mcal}(t)R_\eps u:t\ge0\}
\]
are precompact in $\Mcal_\eps$. Then
\begin{equation}
 d_{H,\Hcal}\!\left(
 \omega(R_\eps S(\cdot)u),
 \omega(S_\eps^{\Mcal}(\cdot)R_\eps u)
 \right)
 \le \Dcal_{\eps,\infty}(u).
\label{eq:omega-limit-semiconjugacy}
\end{equation}
\end{proposition}

\begin{proof}
Set
\[
    x_\eps(t):=R_\eps S(t)u,
    \qquad
    y_\eps(t):=S_\eps^{\Mcal}(t)R_\eps u.
\]
By definition of the uniform-in-time semi-conjugacy,
\[
    \norm{x_\eps(t)-y_\eps(t)}_{\Hcal}
    \le \Dcal_{\eps,\infty}(u),
    \qquad t\ge0.
\]
Let $x_\ast\in\omega(x_\eps)$. 
By definition of the omega-limit set, there exists a sequence $t_n\to\infty$ such that
$x_\eps(t_n)\to x_\ast$. 
Since the forward trajectory of $y_\eps$ is precompact, after passing to a subsequence there exists
$y_\ast\in\Mcal_\eps$ such that \(y_\eps(t_n)\to y_\ast\).
Since $t_n\to\infty$, we have $y_\ast\in\omega(y_\eps)$. 
Passing to the limit in
\[
    \norm{x_\eps(t_n)-y_\eps(t_n)}_{\Hcal}
    \le \Dcal_{\eps,\infty}(u)
\]
gives
\[
    \dist_\Hcal(x_\ast,\omega(y_\eps))
    \le \norm{x_\ast-y_\ast}_{\Hcal}
    \le \Dcal_{\eps,\infty}(u).
\]
Taking the supremum over $x_\ast\in\omega(x_\eps)$ yields one of the two one-sided Hausdorff estimates. 
Interchanging the roles of $x_\eps$ and $y_\eps$ gives the reverse estimate, and hence \eqref{eq:omega-limit-semiconjugacy} follows.
\end{proof}

\subsection{Attractor convergence from uniform-in-time semi-conjugacy}
\label{sec:attractor-convergence-uniform}

\begin{assumption}[Existence of global attractors]
\label{ass:global-attractors}
The original semiflow $S(t)$ has a compact global attractor $\Acal$, and, for all sufficiently small $\eps>0$, $S_\eps(t)$ has a compact global attractor $\Acal_\eps$.
\end{assumption}

\begin{remark}
We note that the above assumption is standard for dissipative parabolic equations. 
In particular, the existence of a bounded absorbing set together with asymptotic compactness of the semiflow implies the existence of a compact global attractor. 
In the present setting, asymptotic compactness is typically provided by the smoothing properties of the parabolic semigroup. 
Thus, once suitable dissipative estimates are available for both the original and the
renormalized equations, Assumption~\ref{ass:global-attractors} can usually be verified by standard attractor theory; see, for example, \cite{Temam1997}.
\end{remark}

For later use, set
\begin{equation}
    \vartheta_\Acal(\eps)
    :=\sup_{a\in\Acal}\norm{Q_\eps a}_{\Hcal}.
    \label{eq:attractor-tail-modulus}
\end{equation}
Strict invariance and boundedness of $\Acal$, together with the high-mode estimate, yield
\begin{equation}
    \vartheta_\Acal(\eps)
    \le\frac{M}{\lambda_\eps^+}
    \le\frac{M}{\mu}\eps^2.
    \label{eq:attractor-tail-bound}
\end{equation}
Indeed, for $a\in\Acal$ and $T>0$, choose $b_T\in\Acal$ with $S(T)b_T=a$. 
Then, by Proposition \ref{prop:tail-dynamics} we obtain
\[
    \norm{Q_\eps a}_{\Hcal}
    \le e^{-\lambda_\eps^+T}
         \sup_{b\in\Acal}\norm{b}_{\Hcal}
       +\frac{M}{\lambda_\eps^+}.
\]
Letting $T\to\infty$ and taking the supremum over $a\in\Acal$ proves \eqref{eq:attractor-tail-bound}.

\begin{theorem}[Hausdorff convergence from uniform-in-time semi-conjugacy]
\label{thm:hausdorff-from-uniform-semiconjugacy}
Assume Assumptions~\ref{ass:operator}, \ref{ass:nonlinearity}, \ref{ass:spectral_shift} and~\ref{ass:global-attractors}.
For small $\eps$, let
\[
    \eta_\eps:=\sup_{u\in\Acal\cup\Acal_\eps}\Dcal_{\eps,\infty}(u),
\]
and suppose $\eta_\eps\leq C_0\eps^2$ with $C_0$ independent of $\eps$.
With $d_\eps=M/\lambda_\eps^++M/\alpha_\eps$ as in \eqref{eq:rho},
\begin{equation}
    d_{H,\Hcal}(\Acal_\eps,\Acal)
    \leq\eta_\eps+d_\eps
    \leq\left(C_0+\frac{2M}{\mu}\right)\eps^2.
    \label{eq:hausdorff-attractor-uniform}
\end{equation}
In particular, the conclusion holds under the hypotheses of Theorem~\ref{thm:uniform-semiconjugacy}.
\end{theorem}

\begin{proof}
The inertial manifold contains $\Acal_\eps$, so $R_\eps b=b$ and
$\norm{Q_\eps b}\leq M/\alpha_\eps$ for $b\in\Acal_\eps$.
The tail and graph bounds imply
\[
    \sup_{a\in\Acal}\norm{a-R_\eps a}\leq d_\eps,
    \quad \text{and} \quad
    \sup_{b\in\Acal_\eps}\norm{S(T)b-R_\eps S(T)b}
    \leq d_\eps+\e^{-\lambda_\eps^+T}\frac{M}{\alpha_\eps}.
\]
Strict invariance of each attractor, the defect bound and the triangle inequality give
\[
\begin{aligned}
    \dist_{\Hcal}(\Acal_\eps,\Acal)
    &\leq\eta_\eps+d_\eps
       +\e^{-\lambda_\eps^+T}\frac{M}{\alpha_\eps}
       +\dist_{\Hcal}(S(T)\Acal_\eps,\Acal),\\
    \dist_{\Hcal}(\Acal,\Acal_\eps)
    &\leq\eta_\eps+d_\eps
       +\dist_{\Hcal}(S_\eps(T)R_\eps\Acal,\Acal_\eps).
\end{aligned}
\]
For each fixed $\eps$, the compact sets $\Acal_\eps$ and $R_\eps\Acal$ are attracted by the respective global attractors, so letting $T\to\infty$ proves \eqref{eq:hausdorff-attractor-uniform}.
Finally, Theorem~\ref{thm:uniform-semiconjugacy} and the tail bounds on $\Acal\cup\Acal_\eps$ give the required $\eta_\eps\leq C_0\eps^2$.
\end{proof}

\begin{remark} \label{rem:uniform-time-attractor-conclusion} 
The preceding proof clarifies the additional dynamical information provided by uniform-in-time semi-conjugacy. 
A finite-time estimate controls the commutation defect only on intervals $[0,T]$, with a constant that may deteriorate as $T\to\infty$. 
In that setting, upper semi-continuity of attractors can still be obtained under an additional uniform boundedness assumption on the perturbed attractors: one first chooses a sufficiently large, but fixed, attraction time $T$ and then lets $\eps\to0$. 
This is the usual one-sided perturbation argument; see, for example, \cite{HoangOlsonRobinson2015,BortolanCarvalhoLanga2020}. 

The uniform-in-time estimate removes this restriction on the order of the limits. 
For fixed $\eps$, one may let $T\to\infty$ while keeping the semi-conjugacy defect uniformly of order $O(\eps^2)$. 
Together with strict invariance of $\Acal_\eps$, this yields the estimate of $\dist_{\Hcal}(\Acal_\eps,\Acal)$ without requiring the family $\{\Acal_\eps\}$ to lie in a common bounded set. 
Conversely, strict invariance of $\Acal$ allows one to pull back points of $\Acal$, while the compact set $R_\eps\Acal\subset\Mcal_\eps$ is attracted by $\Acal_\eps$. 
The same uniform commutation estimate therefore controls $\dist_{\Hcal}(\Acal,\Acal_\eps)$ as well. 
Hence, once the two global attractors exist, uniform-in-time approximate semi-conjugacy upgrades the finite-time dynamical comparison to full Hausdorff convergence of the attractors. 
\end{remark}

\subsection{Regularity of the perturbation and persistence of equilibria}
\label{sec:equilibrium-persistence}

The $C^0$ comparison developed in Section \ref{sec:approximate-semiconjugacy} is sufficient for orbit shadowing and semi-continuity of the attractor, but it cannot by itself preserve linearized stability or the Morse index of equilibria.  
To answer these questions we have to compare the vector fields in a topology in which the artificial term is small.  
The key point is that this becomes possible on the domain of $A$, or more generally on a sufficiently regular level of the Hilbert scale.

\begin{assumption}[Higher regularity near equilibria]
\label{ass:smooth-F-dynamics}
On an open bounded neighborhood $\Ucal\subset\Hcal^2$ of the equilibrium under consideration, the restriction $F|_\Ucal:\Ucal\to\Hcal$ is continuously Fr\'echet differentiable.
For each $u\in\Ucal$, its derivative extends to a bounded operator $DF(u):\Hcal\to\Hcal$.
There are constants $C_\Ucal,L_\Ucal>0$ such that
\begin{equation}
    \sup_{u\in\Ucal}\norm{DF(u)}_{\mathcal L(\Hcal)}\leq C_\Ucal
    \label{eq:DF-H-bound}
\end{equation}
and
\begin{equation}
    \norm{DF(u)-DF(v)}_{\mathcal L(\Hcal)}
    \leq L_\Ucal\norm{u-v}_2,
    \qquad u,v\in\Ucal.
    \label{eq:DF-local-Lipschitz-H2}
\end{equation}
\end{assumption}

\begin{definition}
An equilibrium $u_*\in\Hcal^2$ satisfies
\begin{equation*}
    \mathscr G(u_*):=Au_*-F(u_*)=0.
    \label{eq:stationary-original}
\end{equation*}
It is \emph{nondegenerate} if
\begin{equation*}
    D\mathscr G(u_*)=A-DF(u_*):\Hcal^2\to\Hcal
    \label{eq:stationary-linearization}
\end{equation*}
is an isomorphism, and \emph{hyperbolic} if, after complexification, its dynamical linearization
\begin{equation*}
    \mathscr L_*=-A+DF(u_*),\qquad\dom(\mathscr L_*)=\Hcal^2,
    \label{eq:dynamical-linearization-original}
\end{equation*}
satisfies
\begin{equation*}
    \sigma(\mathscr L_*)\cap i\R=\varnothing.
    \label{eq:hyperbolic-equilibrium}
\end{equation*}
\end{definition}

In particular, hyperbolicity implies nondegeneracy.
Compactness of the resolvent and boundedness of $DF(u_*)$ make the unstable spectral subspace finite-dimensional; its dimension, counting algebraic multiplicities, is the \emph{Morse index} $m(u_*)$.

\begin{theorem}[Persistence of nondegenerate equilibria]
\label{thm:equilibrium-persistence}
Assume Assumptions \ref{ass:operator}, \ref{ass:spectral_shift} and \ref{ass:smooth-F-dynamics}.  
Let $u_*\in\Hcal^r$, $r\geq2$, be a nondegenerate equilibrium of the original equation.
Then there are $\eps_0>0$, a neighborhood $U_*\subset\Hcal^2$ of $u_*$, and a unique family of equilibria
\begin{equation}
    u_*^\eps\in U_* ,
    \qquad
    A_\eps u_*^\eps=F(u_*^\eps),
    \qquad 0<\eps\leq\eps_0,
    \label{eq:continued-equilibrium}
\end{equation}
with
\begin{equation}
    \norm{u_*^\eps-u_*}_2
    \leq C_*\eps^{r-1}\norm{u_*}_r.
    \label{eq:equilibrium-persistence-rate}
\end{equation}
In particular, every $\Hcal^2$ nondegenerate equilibrium persists with an $O(\eps)$ error in the graph norm.
\end{theorem}

\begin{proof}
Define
\[
    \mathscr G_\eps(u)
    :=A_\eps u-F(u)
    =\mathscr G(u)+\eps^{-1}Q_\eps u,
    \qquad u\in\Hcal^2.
\]
By Lemma \ref{lem:tail} with $a=0$ and $r=2$, for every
$z\in\Hcal^2$,
\[
    \eps^{-1}\norm{Q_\eps z}_{\Hcal}
    \le
    \eps^{-1}(1+\Lambda_\eps)^{-1}\norm{z}_2
    \le C\eps\norm{z}_2,
\]
where we use $\Lambda_\eps=\mu\eps^{-2}$. 
Hence
\begin{equation}
    \norm{
        D\mathscr G_\eps(u_*)-D\mathscr G(u_*)
    }_{\mathcal L(\Hcal^2,\Hcal)}
    =
    \norm{\eps^{-1}Q_\eps}_{\mathcal L(\Hcal^2,\Hcal)}
    \le C\eps.
    \label{eq:stationary-derivative-perturbation}
\end{equation}

Set
\[
    L_*:=D\mathscr G(u_*)
    =A-DF(u_*):\Hcal^2\to\Hcal,
    \qquad
    K_*:=L_*^{-1}\in\mathcal L(\Hcal,\Hcal^2).
\]
By the nondegeneracy of the equilibrium $K_*$ is well-defined. 
Since $\mathscr G(u_*)=0$, Lemma \ref{lem:tail} also yields
\begin{equation}
    \norm{\mathscr G_\eps(u_*)}_{\Hcal}
    =
    \eps^{-1}\norm{Q_\eps u_*}_{\Hcal}
    \le C\eps^{r-1}\norm{u_*}_r.
    \label{eq:stationary-residual}
\end{equation}

Write $u=u_*+w$ and define
\[
    \mathcal T_\eps(w)
    :=w-K_*\mathscr G_\eps(u_*+w).
\]
Then $u_*+w$ is a zero of $\mathscr G_\eps$ if and only if $w$ is
a fixed point of $\mathcal T_\eps$. 
Moreover,
\[
    D\mathcal T_\eps(w)
    =
    K_*\Bigl(
        DF(u_*+w)-DF(u_*)-\eps^{-1}Q_\eps
    \Bigr).
\]
By the local Lipschitz estimate \eqref{eq:DF-local-Lipschitz-H2} and \eqref{eq:stationary-derivative-perturbation}, there are $C_F>0$ and $\rho>0$ such that
\[
    \norm{D\mathcal T_\eps(w)}_{\mathcal L(\Hcal^2)}
    \le
    \norm{K_*}\bigl(C_F\rho+C\eps\bigr),
    \qquad \norm{w}_2\le\rho.
\]
Choosing first $\rho>0$ and then $\eps_0>0$ sufficiently small, we obtain
\[
    \sup_{\norm{w}_2\le\rho}
    \norm{D\mathcal T_\eps(w)}_{\mathcal L(\Hcal^2)}
    \le q<1,
    \qquad 0<\eps\le\eps_0.
\]
Furthermore, by \eqref{eq:stationary-residual},
\[
    \norm{\mathcal T_\eps(0)}_2
    \le C\eps^{r-1}\norm{u_*}_r.
\]
After decreasing $\eps_0$ if necessary, this is at most $(1-q)\rho$. 
Hence $\mathcal T_\eps$ maps the closed ball $\overline B_{\Hcal^2}(0,\rho)$ into itself and is a contraction there. 
Banach's fixed-point theorem therefore yields a unique fixed point $w_\eps$ in this ball. Setting
\[
    u_*^\eps:=u_*+w_\eps,
\]
we obtain
\[
\begin{aligned}
    \norm{u_*^\eps-u_*}_2
    &=\norm{w_\eps}_2 \\
    &\le q\norm{w_\eps}_2+\norm{\mathcal T_\eps(0)}_2,
\end{aligned}
\]
and consequently
\[
    \norm{u_*^\eps-u_*}_2
    \le
    \frac{C}{1-q}\eps^{r-1}\norm{u_*}_r.
\]
This proves the desired estimate.
\end{proof}

\begin{theorem}[Preservation of hyperbolicity and Morse index]
\label{thm:morse-index-persistence}
Under the assumptions of Theorem \ref{thm:equilibrium-persistence}, suppose in addition that $u_*$ is hyperbolic. 
Then $u_*^\eps$ is hyperbolic for all sufficiently small $\eps>0$, and the Morse index satisfies
\begin{equation}
    m(u_*^\eps)=m(u_*).
    \label{eq:morse-index-preserved}
\end{equation}
\end{theorem}

\begin{proof}
Let
\[
    \mathscr L_*:=-A+DF(u_*),
    \qquad
    \mathscr L_*^\eps:=-A_\eps+DF(u_*^\eps),
    \qquad
    \dom(\mathscr L_*)=\dom(\mathscr L_*^\eps)=\Hcal^2.
\]
Then
\[
    \mathscr L_*^\eps-\mathscr L_*
    =
    -\eps^{-1}Q_\eps
    +DF(u_*^\eps)-DF(u_*).
\]
By \eqref{eq:stationary-derivative-perturbation},
Theorem~\ref{thm:equilibrium-persistence}, and the local Lipschitz
continuity of $DF$,
\begin{equation}
    \norm{\mathscr L_*^\eps-\mathscr L_*}
         _{\mathcal L(\Hcal^2,\Hcal)}
    \le C\bigl(\eps+\eps^{r-1}\bigr)
    \rightarrow0.
    \label{eq:graph-norm-linearization-convergence}
\end{equation}

Let $z\in\rho(\mathscr L_*)$. 
Since
\[
    (z-\mathscr L_*)^{-1}\in\mathcal L(\Hcal,\Hcal^2),
\]
\eqref{eq:graph-norm-linearization-convergence} and the factorization
\[
    z-\mathscr L_*^\eps
    =
    \left[
        I-(\mathscr L_*^\eps-\mathscr L_*)
          (z-\mathscr L_*)^{-1}
    \right]
    (z-\mathscr L_*)
\]
show, by a Neumann-series argument, that $z\in\rho(\mathscr L_*^\eps)$ for sufficiently small $\eps$ and
\begin{equation}
    (z-\mathscr L_*^\eps)^{-1}
    \rightarrow
    (z-\mathscr L_*)^{-1}
    \qquad\text{in }\mathcal L(\Hcal).
    \label{eq:norm-resolvent-convergence}
\end{equation}
The convergence is uniform for $z$ in compact subsets of $\rho(\mathscr L_*)$.

Both $\mathscr L_*$ and $\mathscr L_*^\eps$ have compact resolvent.
Since $\mathscr L_*$ is hyperbolic, its unstable spectrum consists of finitely many isolated eigenvalues of finite algebraic multiplicity.
Choose disjoint positively oriented contours contained in the open right half-plane whose union $\Gamma$ encloses precisely these eigenvalues.
By the uniform resolvent convergence on $\Gamma$, the corresponding Riesz projections
\[
    \Pi^u
    :=
    \frac{1}{2\pi i}
    \int_\Gamma (z-\mathscr L_*)^{-1}\,dz,
    \qquad
    \Pi_\eps^u
    :=
    \frac{1}{2\pi i}
    \int_\Gamma (z-\mathscr L_*^\eps)^{-1}\,dz
\]
satisfy
\[
    \norm{\Pi_\eps^u-\Pi^u}_{\mathcal L(\Hcal)}
    \rightarrow0.
\]
Hence, for sufficiently small $\eps$,
\[
    \rank\Pi_\eps^u=\rank\Pi^u.
\]

It remains to exclude additional spectrum in the closed right half-plane. 
Since $u_*^\eps\to u_*$, there is $R>0$ such that
\[
    \norm{DF(u_*^\eps)}_{\mathcal L(\Hcal)}\le R
\]
for all sufficiently small $\eps$. 
As $A_\eps$ is nonnegative and self-adjoint, a Neumann-series argument applied to $z+A_\eps-DF(u_*^\eps)$ shows that
\[
    \sigma(\mathscr L_*^\eps)\cap\{\Re z\ge0\}
    \subset \{z:\Re z\ge0,\ |z|\le R\}.
\]
After slightly enlarging $R$ if necessary, the compact set
\[
    \{z:\Re z\ge0,\ |z|\le R\}
\]
minus the interiors of the contours $\Gamma$ is contained in $\rho(\mathscr L_*)$. 
Uniform resolvent convergence on this compact set therefore excludes any additional spectrum of $\mathscr L_*^\eps$ there for sufficiently small $\eps$.

Consequently, $\mathscr L_*^\eps$ has no spectrum on the imaginary axis and is therefore hyperbolic. 
Moreover, its unstable spectrum is precisely the spectrum enclosed by $\Gamma$, so
\[
    m(u_*^\eps)
    =\rank\Pi_\eps^u
    =\rank\Pi^u
    =m(u_*).
\]
\end{proof}

\begin{remark}
\label{rem:active-modes-cutoff}
For a fixed hyperbolic equilibrium, the unstable spectrum is finite.  
Since the cutoff satisfies $\Lambda_\eps\to\infty$, the artificial spectral gap is eventually placed far beyond every fixed low-frequency scale.  
The proof above makes this principle quantitative without requiring the linearized eigenvectors to coincide with eigenvectors of $A$.  
This is the rigorous version of the requirement that the cutoff should lie beyond the dynamically active spectral range, when we apply this approach in a concrete setting.  
\end{remark}


\section{A reaction-diffusion system without a natural spectral gap}
\label{sec:coupled-AC-benchmark}

In this section we apply the renormalization procedure to a specific example that satisfies the following properties:
\begin{enumerate}[label=\textup{(\roman*)},leftmargin=2.2em]
    \item the original equation has nontrivial dynamics, including, for example, several equilibria of different Morse index and spatially heterogeneous steady states;
    \item the natural spectrum does not satisfy the spectral-gap condition required by the standard inertial-manifold theorem at any spectral interface; and 
    \item the artificial spectral renormalization applies without any problem-specific modification.
\end{enumerate}
As a model that meets these criteria we choose a coupled Allen-Cahn system on a three-dimensional bounded domain.\\

The Allen-Cahn equation originates in the classical phase-field model of Allen and Cahn \cite{AllenCahn1979}.  
Vector-valued and multi-component Allen-Cahn equations associated with multi-well potentials have subsequently been studied in connection with multiphase interfaces, heteroclinic structures, and more complicated stationary patterns; see, for example, \cite{BronsardReitich1993,AlamaBronsardGui1997, BatesFuscoSmyrnelis2017,Fusco2017}.
Of particular relevance to our example are symmetric linearly coupled Allen-Cahn systems, for which the coupling changes both the stationary states and their spectral stability; see, for example, \cite{Fazly2025LinearCoupling}.

\paragraph{Problem set-up.}
On $\Omega=(0,\pi)^3$, we consider the following system
\begin{equation}
\begin{aligned}
    \partial_tu&=\Delta u+\rho u-u^3+b(v-u),\\
    \partial_tv&=\Delta v+\rho v-v^3+b(u-v),\\
    \partial_\nu u&=\partial_\nu v=0\qquad\text{on }\partial\Omega.
\end{aligned}
\label{eq:coupled-AC-rho}
\end{equation}
To fix the parameters, we choose
\begin{equation}
    \rho=1+\theta,\qquad b=\frac25+\theta,\qquad -\theta_*<\theta<\theta_*.
    \label{eq:parameter_regime}
\end{equation}
We let $0<\bar\theta<1/20$ and take $\theta_*\leq\bar\theta$ sufficiently small.
This allows us, in particular, to study the bifurcation through $\theta=0$.

\begin{remark}
The first observation we make is the following.
At the trivial equilibrium zero, the reaction matrix is given by
\begin{equation}
    DF_0(0)=\begin{pmatrix}\rho-b&b\\b&\rho-b\end{pmatrix}.
    \label{eq:coupled-AC-linearization-zero-reaction}
\end{equation}
Its synchronized and anti-synchronized eigenvectors are $e_{\rm s}=2^{-1/2}(1,1)^T$ and $e_{\rm a}=2^{-1/2}(1,-1)^T$, with corresponding eigenvalues
\begin{equation}
    \sigma_{\rm s}=\rho=1+\theta,\qquad
    \sigma_{\rm a}=\rho-2b=\frac15-\theta.
    \label{eq:coupled-AC-sync-antisync-eigenvalues}
\end{equation}
Thus the synchronized mode has crossed the first nonzero Neumann eigenvalue $\lambda_N=1$, whereas the anti-synchronized instability is confined to the spatially constant mode.
\end{remark}

We write $U=(u,v)^T$, set $\Hcal=L^2(\Omega;\R^2)$ as underlying function space, and define
\begin{equation}
    A=\begin{pmatrix}-\Delta_N&0\\0&-\Delta_N\end{pmatrix},
    \label{eq:coupled-AC-A}
\end{equation}
\begin{equation}
    F_0(U)=\begin{pmatrix}(\rho-b)u-u^3+bv\\(\rho-b)v-v^3+bu\end{pmatrix}.
    \label{eq:coupled-AC-F0}
\end{equation}
The original system can be now expressed as
\begin{equation}
    \partial_tU+AU=F_0(U).
    \label{eq:coupled-AC-abstract}
\end{equation}
We note, that the linear operator satisfies Assumption~\ref{ass:operator}, however the nonlinear reaction still requires a modification to satisfy Assumption~\ref{ass:nonlinearity}.

To get a better idea of the possible growth of solutions we show that system \eqref{eq:coupled-AC-abstract} has a compact global attractor.
Indeed, the original system is a gradient system with energy
\begin{equation}
    \Ecal(U)=\int_\Omega\left[
       \frac12(|\nabla u|^2+|\nabla v|^2)
       -\frac\rho2(u^2+v^2)+\frac14(u^4+v^4)
       +\frac b2(u-v)^2\right]dx
    \label{eq:coupled-AC-energy-rho}
\end{equation}
and, along classical solutions,
\begin{equation}
    \frac{d}{dt}\Ecal(U(t))
    =-\|\partial_tu\|_{L^2}^2-\|\partial_tv\|_{L^2}^2.
    \label{eq:coupled-AC-energy-decay}
\end{equation}
The cubic damping yields a global $L^2$ semiflow.
After positive-time smoothing due to the parabolic operator, the maximum principle bounds both components of the amplitudes by the solution of the equation $\dot a=\rho a-a^3$.
Consequently every box $|u|,|v|<R$, with $R>\sqrt{1+\bar\theta}$, absorbs bounded subsets of $\Hcal$, uniformly for $|\theta|\leq\bar\theta$.
Using the parabolic smoothing then proves the existence of a compact global attractor $\Acal$.

\paragraph{A bounded dissipative cutoff.}
The existence of a compact global attractor motivates our choice for the cutoff function.
First, we fix such an $R>\sqrt{1+\bar\theta}>0$ and choose $R_1>R_0>\sqrt2R$, and then take a smooth function $\chi:[0,\infty)\to[0,1]$ equal to one on $[0,R_0^2]$ and zero on $[R_1^2,\infty)$.
Writing $f_0:\R^2\to\R^2$ for the nonlinear reaction in \eqref{eq:coupled-AC-F0}, we set
\[
    f(z)=\chi(|z|^2)f_0(z)
          -(1-\chi(|z|^2))\frac{z}{\sqrt{1+|z|^2}},
    \qquad F(U)(x)=f(U(x)).
\]
Thus, this cutoff composed with the original nonlinearity is smooth in $z$ and $\theta$, agrees with $f_0$ near $[-R,R]^2$, and has uniformly bounded derivatives for $|\theta|\leq\bar\theta$.
Its Nemytskii map satisfies
\begin{equation}
    \norm{F(U)}_{\Hcal}\leq M,\qquad
    \norm{F(U)-F(V)}_{\Hcal}\leq L\norm{U-V}_{\Hcal},
    \label{eq:coupled-AC-prepared-F}
\end{equation}
with $M,L$ independent of $\theta$ in this interval.

Then, we introduce 
\[
    \Ucal=\{(u,v)\in L^\infty(\Omega;\R^2):
       \|u\|_\infty<R,\ \|v\|_\infty<R\}.
\]
Both equations leave this set positively invariant and absorb bounded $L^2$ sets into it.
Indeed, on a face where $z_i=\pm a$ and $|z_{3-i}|\leq a$, $a\geq R$, the prepared reaction satisfies
\[
    \operatorname{sgn}(z_i)f_i(z)
    \leq-\chi(|z|^2)a(a^2-\rho)
       -(1-\chi(|z|^2))\frac{a}{\sqrt{1+2a^2}}
    \leq-c_R<0.
\]
Using the comparison principle and positive-time smoothing we obtain the entrance property.
In particular, our chosen modification preserves the original attractor $\Acal$, as the two semiflows have the same bounded complete trajectories, all of which lie in $\Ucal$.

\begin{remark}
\label{rmk:modified-original}
Let $\widetilde S(t)$ and $S(t)$ denote the polynomial and prepared semiflows, respectively.
For bounded $B\subset\Hcal$, choose $t_B$ so that $\widetilde S(t_B)B\subset\Ucal$.
Uniqueness and positive invariance give
\[
    \widetilde S(t)U_0
    =S(t-t_B)\widetilde S(t_B)U_0,
    \qquad U_0\in B,\quad t\geq t_B.
\]
Thus comparisons transfer to the polynomial equation from the entry state, and apply directly to its equilibria and complete trajectories on $\Acal$.
The two flows need not agree before this restart.
\end{remark}

\begin{remark}
On the region $\Ucal$ the reaction derivative is given by
\begin{equation}
    DF_0(u,v)=
    \begin{pmatrix}\rho-b-3u^2&b\\b&\rho-b-3v^2\end{pmatrix}.
    \label{eq:coupled-AC-matrix-derivative}
\end{equation}
This is a matrix-valued multiplication operator, so the scalar spatial-averaging criterion discussed in Section~\ref{subsec:three-dimensional-spatial-averaging} does not apply automatically.
However, we want to point out that failure of this criterion and of the ordinary spectral-gap condition does not establish nonexistence of an inertial manifold for the original system.
\end{remark}

\paragraph{Renormalized system.}
In the next step we apply the spectral cutoff and shifted operator from Assumption~\ref{ass:spectral_shift} to define
\begin{equation}
    \partial_tU^\eps+AU^\eps+\eps^{-1}Q_\eps U^\eps=F(U^\eps).
    \label{eq:coupled-AC-renormalized}
\end{equation}
Whenever $U^\eps$ lies in $\Ucal$, this can be written as
\begin{equation}
\begin{aligned}
    \partial_tu^\eps&=\Delta u^\eps+(\rho-b)u^\eps-(u^\eps)^3+bv^\eps
                         -\eps^{-1}(Q_\eps U^\eps)_1,\\
    \partial_tv^\eps&=\Delta v^\eps+(\rho-b)v^\eps-(v^\eps)^3+bu^\eps
                         -\eps^{-1}(Q_\eps U^\eps)_2.
\end{aligned}
\label{eq:coupled-AC-renormalized-components}
\end{equation}

For the renormalized equation, dissipativity follows from an $L^2$ estimate.
Since 
$$f_0(z)\cdot z\leq\rho|z|^2-|z|^4/2\quad \text{and}\quad R_0>\sqrt{2\rho},$$
the nonlinear cutoff satisfies $f(z)\cdot z\leq C-c|z|$.
Hence
\[
    \frac12\frac{d}{dt}\norm{U^\eps}_{\Hcal}^2
    +\norm{\nabla U^\eps}_{L^2}^2
    +\eps^{-1}\norm{Q_\eps U^\eps}_{\Hcal}^2
    \leq C-c\norm{U^\eps}_{L^1}.
\]
The Neumann Poincar\'e inequality gives
$\frac{d}{dt}\norm{U^\eps}_{\Hcal}^2\leq C_1-c_1\norm{U^\eps}_{\Hcal}$.
Thus the renormalized semiflows have a common bounded absorbing set in $\Hcal$.
Compact parabolic smoothing yields global attractors $\Acal_\eps$, verifying Assumption~\ref{ass:global-attractors} for this example.

\begin{proposition}[Renormalized inertial manifold]
\label{prop:coupled-AC-inertial-manifold}
For all sufficiently small $\eps>0$, system \eqref{eq:coupled-AC-renormalized} has an inertial manifold
\begin{equation}
    \Mcal_\eps=\{p+h_\eps(p):p\in P_\eps\Hcal\},
    \label{eq:coupled-AC-inertial-manifold}
\end{equation}
with
\begin{equation}
    \norm{h_\eps}_\infty\leq C\eps^2,\qquad \Lip(h_\eps)\leq C\eps,
    \label{eq:coupled-AC-graph-estimates}
\end{equation}
and $\dim P_\eps\Hcal=O(\eps^{-3})$.
\end{proposition}

\begin{proof}
The artificial gap satisfies $\gamma_\eps\geq\eps^{-1}$, so Theorem~\ref{thm:inertial-manifold} applies for small $\eps$.
Counting $k\in\N_0^3$ with $|k|^2\leq\Lambda_\eps$, for each component, gives the dimension estimate.
\end{proof}

For bounded $B\subset\Hcal$ and fixed $T>0$, we apply Theorem~\ref{thm:finite-semiconjugacy} and obtain
\begin{equation}
    \sup_{U_0\in B}\sup_{0\leq t\leq T}
    \norm{R_\eps S(t)U_0-S_\eps^\Mcal(t)R_\eps U_0}_{\Hcal}
    \leq C_{B,T}\eps^2.
    \label{eq:coupled-AC-finite-time-semiconjugacy}
\end{equation}
For fixed $\kappa\geq2$, set $\tau_\eps^\kappa=(\kappa/\mu)\eps^2|\log\eps|$.
Then, by Theorem~\ref{thm:post-layer} we also have ,for small $\eps$ with $\tau_\eps^\kappa\leq T$,
\begin{equation}
    \sup_{U_0\in B}\sup_{\tau_\eps^\kappa\leq t\leq T}
    \norm{S(t)U_0-S_\eps^\Mcal(t-\tau_\eps^\kappa)
          R_\eps S(\tau_\eps^\kappa)U_0}_{\Hcal}
    \leq C_{B,T,\kappa}\eps^2.
    \label{eq:coupled-AC-post-layer}
\end{equation}
Note that in both results we compare the renormalized equation  with the exact graph dynamics of the full PDE after the nonlinear cutoff.

\paragraph{Homogeneous equilibria.}
A constant equilibrium of the original system satisfies
\begin{equation}
    (\rho-b)u-u^3+bv=0,\qquad (\rho-b)v-v^3+bu=0.
    \label{eq:coupled-AC-homogeneous-equations}
\end{equation}

\begin{proposition}[Homogeneous equilibria]
\label{prop:coupled-AC-homogeneous-equilibria}
In the parameter regime \eqref{eq:parameter_regime}, the original nonlinear system and the one with the nonlinear cutoff have exactly five homogeneous equilibria,
\begin{equation}
    E_0=(0,0),\qquad E_{\rm s}^{\pm}=\pm\sqrt\rho(1,1),\qquad
    E_{\rm a}^{\pm}=\pm\sqrt{\rho-2b}(1,-1).
    \label{eq:coupled-AC-five-equilibria}
\end{equation}
All are hyperbolic for the full PDE, with Morse indices
\begin{equation}
    m(E_0)=5,\qquad m(E_{\rm a}^{\pm})=1,\qquad m(E_{\rm s}^{\pm})=0.
    \label{eq:coupled-AC-Morse-indices}
\end{equation}
\end{proposition}

\begin{proof}
The cases $u=v$ and $u=-v$ give the displayed equilibria.
Otherwise, adding and subtracting \eqref{eq:coupled-AC-homogeneous-equations} gives
$u^2+v^2=\rho-b=3/5$ and $uv=-b$, contradicting $u^2+v^2\geq2b>4/5$.
The modification via the nonlinear cutoff adds no constant equilibria because $f(z)\cdot z<0$ for $|z|>R_0$ and it agrees with $f_0$ for $|z|\leq R_0$.
To be more precise, we have the following:
At $E_0$, the reaction eigenvalues in the synchronized and antisynchronized directions are $1+\theta$ and $1/5-\theta$, respectively.
At $E_{\rm a}^{\pm}$ they are $2/5+4\theta$ and $-2(1/5-\theta)$, while at $E_{\rm s}^{\pm}$ they are $-2\rho$ and $-2(\rho+b)$, in the same order.
We obtain the growth rates of the full PDE by subtracting the Neumann eigenvalues: $0$ is simple, $1$ has multiplicity three, and all remaining levels are at least $2$.
In our chosen parameter regime \eqref{eq:parameter_regime}, this gives five positive growth rates at $E_0$, one at each $E_{\rm a}^{\pm}$, and none at $E_{\rm s}^{\pm}$, proving \eqref{eq:coupled-AC-Morse-indices}.
Moreover, no growth rate vanishes, so all five equilibria are hyperbolic.
\end{proof}

\begin{remark}
Let us briefly explain why Laplacian on a three dimensional cube has no growing consecutive spectral gaps.
The distinct Neumann eigenvalues are the nonnegative integers representable as sums of three squares.
The missing integers have the form $4^a(8\ell+7)$ and hence residues $0$, $4$ or $7$ modulo $8$ \cite{HardyWright2008}.
Among any three consecutive integers at least one is represented, so adjacent distinct eigenvalues differ by at most $3$; multiplicities add only zero gaps.
Comparing nearby constant states at $E_{\rm s}^{\pm}$ gives
\[
    L\geq\norm{Df(E_{\rm s}^{\pm})}=2(\rho+b)>14/5.
\]
Every consecutive gap is therefore smaller than $2L$.
Both the classical $2L$ spectral gap criterion and the stronger condition used here fail at every interface for this operator.
\end{remark}

\paragraph{Uniform comparison near the synchronized sinks.}
At $E=E_{\rm s}^{\pm}$ the Jacobian of the reaction term is symmetric, with eigenvalues $-2\rho$ and $-2(\rho+b)$, so
\begin{equation}
    \sup\Re\sigma(\mathscr L_E)\leq-2\rho<0.
    \label{eq:coupled-AC-sink-spectral-margin}
\end{equation}
We then choose a small convex neighborhood $\Ucal_\pm\subset\R^2$ of $E$ on which
\begin{equation}
    (f(z_1)-f(z_2))\cdot(z_1-z_2)
    \leq-c_0|z_1-z_2|^2,
    \qquad z_1,z_2\in\Ucal_\pm.
    \label{eq:coupled-AC-local-monotonicity}
\end{equation}
Let $\widehat\Ucal_\pm$ denote the functions taking values in $\Ucal_\pm$ almost everywhere, and define
\begin{equation}
    V_\eps^\pm=\{p\in P_\eps\Hcal:p+h_\eps(p)\in\widehat\Ucal_\pm\}.
    \label{eq:stable-reduced-region}
\end{equation}
For $p_1,p_2\in V_\eps^\pm$, we set $d=p_1-p_2$, $k=h_\eps(p_1)-h_\eps(p_2)$, and $Z_i=p_i+h_\eps(p_i)$.
Orthogonality and $\ell_\eps:=\Lip(h_\eps)=O(\eps)$ give
\[
\begin{aligned}
    \langle P_\eps(F(Z_1)-F(Z_2)),d\rangle
    &\leq-c_0\norm{d+k}^2+L\norm{d+k}\norm{k}\leq\left(-c_0+L\ell_\eps\sqrt{1+\ell_\eps^2}\right)\norm{d}^2.
\end{aligned}
\]
Since $A\geq0$, Assumption~\ref{ass:local-reduced-contraction} holds on $V_\eps^\pm$ with $\sigma=c_0/2$ for small $\eps$.

It remains to verify entrance into these regions for both curves used in that assumption.
Fix a compact set $K\subset\Bcal(E):=\{U_0\in\Hcal:S(t)U_0\to E\}$ and $s\in(3/2,2)$.
The Neumann scale $\Hcal^s$ is an algebra and embeds into $L^\infty$; on the cube this follows by even periodic extension.
From the variation of constants formula we obtain
\[
    \sup_p\norm{h_\eps(p)}_s
    \leq C_sM\int_0^\infty t^{-s/2}\e^{-\lambda_\eps^+t/2}\,dt
    \leq C_s\eps^{2-s}.
\]
Moreover $Q_\eps E=0$, so $E$ is an equilibrium of the renormalized equation and $h_\eps(E)=0$.
Since $DF(E)$ commutes with $A_\eps$, we have
\[
    \norm{\e^{t(-A_\eps+DF(E))}}_{\mathcal L(\Hcal^s)}\leq\e^{-2\rho t}.
\]
Thus, the nonlinear remainder is bounded by $C\norm{U-E}_s^2$ in $\Hcal^s$ near $E$.
Variation of constants therefore gives a stable ball about $E$ in $\Hcal^s$, with radius independent of $\eps$, contained in $\widehat\Ucal_\pm$.
Note that the same ball construction applies to the original equation.

Compactness of $K$, smoothing and local stability give a common time $t_*$ after which $S(t)K$ stays in such a small ball.
Theorem~\ref{thm:finite-semiconjugacy} and the spectral inverse estimate yield
\[
    \sup_{U_0\in K}
    \norm{P_\eps S(t_*)U_0-\Phi_\eps(t_*)P_\eps U_0}_s
    \leq C_{K,t_*}(1+\Lambda_\eps)^{s/2}\eps^2
    \leq C_{K,t_*}\eps^{2-s}.
\]
Together with the graph bound, this places the reduced trajectories in the uniform stable ball at $t_*$.
They remain there, while $P_\eps S(t)U_0+h_\eps(P_\eps S(t)U_0)$ also stays in $\widehat\Ucal_\pm$ by the smaller-ball choice and the uniform graph bound.
Thus both retained curves lie in $V_\eps^\pm$ for $t\geq t_*$ and we can apply Theorem~\ref{thm:stable-region-semiconjugacy}, which yields
\[
    \sup_{U_0\in K}\sup_{t\geq0}
    \norm{R_\eps S(t)U_0-S_\eps^\Mcal(t)R_\eps U_0}_{\Hcal}
    \leq C_K\eps^2.
\]

\begin{remark}
Compactness inside a basin is a sufficient condition for a common entrance time.
Such a time need not exist for bounded sets approaching the basin boundary, where trajectories may remain near saddles for arbitrarily long intervals.
\end{remark}

\paragraph{Spatially heterogeneous equilibria.}
Here, we consider the invariant diagonal subspace
\begin{equation}
    \Dcal=\{(u,v)\in\Hcal:u=v\}
    \label{eq:coupled-AC-diagonal-subspace}
\end{equation}
Setting $u=v=w$ in the coupled Allen-Cahn equation yields
\begin{equation}
    \partial_tw=\Delta w+(1+\theta)w-w^3.
    \label{eq:coupled-AC-diagonal-Allen-Cahn}
\end{equation}
We further restrict to functions depending only on $x_1$, removing the multiplicity of the first spatial eigenspace.
The equation for equilibria in this case is
\begin{equation}
    \phi''+(1+\theta)\phi-\phi^3=0,
    \qquad\phi'(0)=\phi'(\pi)=0,
    \label{eq:coupled-AC-pattern-equation}
\end{equation}
where its operator domain is $H_N^2(0,\pi)=\{\phi\in H^2(0,\pi):\phi'(0)=\phi'(\pi)=0\}$.

\begin{proposition}[Small-amplitude patterned branch]
\label{prop:coupled-AC-pattern-branch}
For sufficiently small $0<\theta<\theta_*$ equation \eqref{eq:coupled-AC-pattern-equation} has two nonconstant solutions satisfying
\begin{equation}
    \phi_\theta^\pm(x)=\pm\frac{2}{\sqrt3}\sqrt\theta\cos x+O(\theta^{3/2})
    \label{eq:coupled-AC-pattern-expansion}
\end{equation}
in $H^2(0,\pi)$, and hence in $C^1([0,\pi])$.
Consequently,
\begin{equation}
    U_{\theta,1}^{\pm}(x)
    =(\phi_\theta^\pm(x_1),\phi_\theta^\pm(x_1))^T
    \label{eq:coupled-AC-pattern-x1}
\end{equation}
are equilibria of the full coupled system.
Permuting coordinates gives the $x_2$- and $x_3$-dependent branches.
\end{proposition}

\begin{proof}
Write $e(x)=\cos x$ and $\phi=ae+\psi$, with $\psi\perp e$.
The operator $\partial_x^2+1$ is invertible on $e^\perp$ with Neumann boundary conditions.
The complementary equation therefore gives a smooth map $\psi(a,\theta)=O(a^3)$ in $H^2$; the term $\theta ae$ has zero complementary projection.
Projection onto $e$ gives
\[
    0=\theta a-
      \frac{\int_0^\pi e^4\,dx}{\int_0^\pi e^2\,dx}a^3+O(a^5)
     =\theta a-\frac34a^3+O(a^5).
\]
Oddness of the equation and the implicit function theorem in $a^2$ yield
$a=\pm2\sqrt{\theta/3}+O(\theta^{3/2})$.
Together with $\psi=O(a^3)$, this proves the expansion.
\end{proof}

The full three-dimensional bifurcation may have additional branches.
The invariant one-coordinate subspaces suffice for the heterogeneous equilibria used here.

\paragraph{Persistence under the spectral renormalization.}
Now, we want to study how the different equilibria persist under our spectral renormalization method.
We observe that every homogeneous state lies in the zero eigenspace of the Neumann Laplacian, so
\begin{equation}
    Q_\eps E_*=0.
    \label{eq:coupled-AC-Q-constant}
\end{equation}

\begin{proposition}[Exact persistence of the homogeneous equilibria]
\label{prop:coupled-AC-exact-persistence}
The five states in \eqref{eq:coupled-AC-five-equilibria} remain equilibria for every $0<\eps\leq1$.
Whenever $\Lambda_\eps\geq1$, they are hyperbolic and their renormalized Morse indices satisfy
\begin{equation}
    m_\eps(E_0)=5,\qquad m_\eps(E_{\rm a}^{\pm})=1,\qquad m_\eps(E_{\rm s}^{\pm})=0.
    \label{eq:coupled-AC-renormalized-indices}
\end{equation}
\end{proposition}

\begin{proof}
Equation \eqref{eq:coupled-AC-Q-constant} preserves the equilibria.
Their constant reaction matrices commute with $Q_\eps$, so the high-mode growth rates are shifted by $-\eps^{-1}$.
Every unstable mode in Proposition~\ref{prop:coupled-AC-homogeneous-equilibria} has spatial eigenvalue $0$ or $1$ and is retained when $\Lambda_\eps\geq1$.
All discarded modes are already strictly stable.
\end{proof}

For heterogeneous equilibria, the local hypotheses of Section~\ref{sec:equilibrium-persistence} hold because the modification via the smooth nonlinear cutoff satisfies
\[
    DF(U)V=Df(U(\cdot))V,
    \qquad
    \norm{DF(U)-DF(V)}_{\mathcal L(\Hcal)}
    \leq C\norm{U-V}_{L^\infty}\leq C\norm{U-V}_2,
\]
and $F:\Hcal^2\to\Hcal$ is continuously Fr\'echet differentiable.
Hence a nondegenerate equilibrium $U_*\in\Hcal^r$, for $r\geq2$, has a unique nearby continuation with
\begin{equation}
    \norm{U_*^\eps-U_*}_2\leq C_*\eps^{r-1}\norm{U_*}_r.
    \label{eq:coupled-AC-pattern-persistence}
\end{equation}

\begin{proposition}[Hyperbolicity and persistence of the patterned states]
\label{prop:coupled-AC-pattern-persistence}
After possibly decreasing $\theta_*$, the equilibria $U_{\theta,j}^{\pm}$, $j\in\{1,2,3\}$, are hyperbolic with Morse index $2$ for every $0<\theta<\theta_*$.
For each fixed such $\theta$ and sufficiently small $\eps$, their unique nearby renormalized continuations satisfy
\[
    \norm{U_{\theta,j}^{\pm,\eps}-U_{\theta,j}^{\pm}}_2
    \leq C_{\theta,r}\eps^{r-1}\norm{U_{\theta,j}^{\pm}}_r,
    \qquad r\geq2,
\]
and $m(U_{\theta,j}^{\pm,\eps})=2$.
In particular, for every fixed $N\in\N$ the error is bounded by $C_{\theta,N}\eps^N$.
\end{proposition}

\begin{proof}
It suffices to consider $(\phi_\theta(x_1),\phi_\theta(x_1))$, with
$3\phi_\theta^2=4\theta\cos^2x_1+O(\theta^2)$ in $L^\infty$.
The linearization splits into the self-adjoint blocks
\[
\begin{aligned}
    \mathscr L_{{\rm s},\theta}
       &=\Delta+1+\theta(1-4\cos^2x_1)+O(\theta^2),\\
    \mathscr L_{{\rm a},\theta}
       &=\Delta+\tfrac15-\theta(1+4\cos^2x_1)+O(\theta^2).
\end{aligned}
\]
At $\theta=0$, the synchronized block has a simple positive eigenvalue $1$.
Its kernel is spanned by $\cos x_1$, $\cos x_2$ and $\cos x_3$, and its remaining spectrum lies in $(-\infty,-1]$.
Compression of $1-4\cos^2x_1$ to this kernel is $\operatorname{diag}(-2,-1,-1)$ in the normalized cosine basis.
Thus the three eigenvalues issuing from zero are $-2\theta+O(\theta^2)$ and $-\theta+O(\theta^2)$, twice, and this block retains exactly one positive eigenvalue.
The anti-synchronized block has one positive eigenvalue near $1/5$ and all others near $(-\infty,-4/5]$.
Both blocks are therefore hyperbolic for small positive $\theta$, giving Morse index $2$.
Sign and coordinate symmetries give the other cases.

Theorems~\ref{thm:equilibrium-persistence} and~\ref{thm:morse-index-persistence} now apply.
The even periodic extension of each Neumann profile is smooth, so it belongs to every $\Hcal^r$.
Taking $r=N+1$ gives the final estimate.
\end{proof}

\begin{proposition}[Persistence of the pattern-forming pitchfork]
\label{prop:coupled-AC-pitchfork-persistence}
For all sufficiently small $\eps>0$, the renormalized equation has a supercritical pitchfork at $\theta=0$ in each invariant synchronized one-coordinate subspace.
The patterned branches satisfy
\[
    \phi_{\theta,\eps}^{\pm}(x)
    =\pm\frac{2}{\sqrt3}\sqrt\theta\cos x+O(\theta^{3/2})
\]
in $H^2(0,\pi)$, uniformly for small $\eps$.
For each fixed $0<\theta<\theta_*$ and sufficiently small $\eps$, these branches agree with the continuations in Proposition~\ref{prop:coupled-AC-pattern-persistence}.
\end{proposition}

\begin{proof}
Near zero, the synchronized one-coordinate stationary equation is
\[
    \phi''+(1+\theta)\phi-\phi^3-\eps^{-1}Q_\eps^{(1)}\phi=0.
\]
Here $Q_\eps^{(1)}$ projects onto the Neumann modes $\cos(nx)$ with $n^2>\Lambda_\eps$.
If $\Lambda_\eps\geq1$, its linearization at $\theta=0$ still has kernel $\operatorname{span}\{\cos x\}$.
On the orthogonal complement, the eigenvalues are $1$ for the constant mode and
$1-n^2-\eps^{-1}\mathbf1_{\{n^2>\Lambda_\eps\}}$ for $n\geq2$.
The inverse from $L^2$ to $H_N^2$ is therefore bounded uniformly in $\eps$.
The complementary equation gives $\psi_\eps(a,\theta)=O(a^3)$ uniformly, and projection onto $\cos x$ yields
\[
    \theta a-\frac34a^3+O(a^5)=0.
\]
This proves the uniform pitchfork expansion.
For fixed $\theta>0$, convergence of the complementary inverses in $\mathcal L(L^2,H_N^2)$ follows from
$\norm{\eps^{-1}Q_\eps^{(1)}}_{\mathcal L(H_N^2,L^2)}=O(\eps)$.
The resulting branches converge in $H^2$ to the original profiles, so local uniqueness identifies them with the hyperbolic continuations.
\end{proof}

\section{Embedding into a two-parameter fast-slow PDE system}
\label{sec:fast-slow}

In this section we apply the spectral renormalization to a fast-slow system, where we do not have to make any assumptions about the spectrum of the slow operator.
This is an important observation as it now allows us to consider systems that were out of reach previously.

The two parameters of the system are time scale separation parameter $\delta$ and the spectral renormalization parameter $\eps$, which are independent elements of $(0,1]$.

Let $\Xcal$ and $\Hcal$ be real Hilbert spaces.
\begin{assumption}[Fast operator]
\label{ass:fast_operator}
The operator $A_f:\dom(A_f)\subset\Xcal\to\Xcal$ is self-adjoint, positive and satisfies $A_f\geq\omega_f I$ for some $\omega_f>0$.
In particular,
\begin{equation}
    \norm{\e^{-tA_f}}_{\mathcal L(\Xcal)}\leq\e^{-\omega_f t},\qquad t\geq0.
    \label{eq:B-fast-semigroup}
\end{equation}
\end{assumption}
In this section we relabel the operator \(A\) of Assumption \ref{ass:operator} as \(A_s\), to indicate the slow operator and to distinguish it from the fast operator \(A_f\).

We consider the fast-slow system
\begin{equation}
\begin{aligned}
    \partial_t u^{\del}+\del^{-1}A_f u^{\del}
        &=f(u^{\del},v^{\del}),\\
    \partial_t v^{\del}+A_sv^{\del}
        &=g(u^{\del},v^{\del}),
\end{aligned}
\label{eq:original-fast-slow}
\end{equation}
and its spectrally renormalized counterpart
\begin{equation}
\begin{aligned}
    \partial_t u^{\del,\eps}+\del^{-1} A_f u^{\del,\eps}
        &=f(u^{\del,\eps},v^{\del,\eps}),\\
    \partial_t v^{\del,\eps}+A_{s,\eps}v^{\del,\eps}
        &=g(u^{\del,\eps},v^{\del,\eps}),
\end{aligned}
\label{eq:renormalised-fast-slow}
\end{equation}
where $A_{s,\eps}=A_s+\eps^{-1}Q_\eps$ is the operator from equation \eqref{eq:renormalized_operator}.  
We use the product norm
\begin{equation*}
    \norm{(u,v)}_{\Xcal\times\Hcal}
    :=\norm{u}_{\Xcal}+\norm{v}_{\Hcal}.
\end{equation*}

\begin{assumption}[Nonlinearities]
\label{ass:coupling}
The maps $f:\Xcal\times\Hcal\to\Xcal$ and $g:\Xcal\times\Hcal\to\Hcal$ are globally bounded and globally Lipschitz in the above product norm.
Set
\[
    M_f=\sup_{u,v}\norm{f(u,v)}_{\Xcal},\qquad
    M_g=\sup_{u,v}\norm{g(u,v)}_{\Hcal},\qquad
    L_*=\max\{\Lip(f),\Lip(g)\}.
\]
\end{assumption}

We write
\[
    v=p+q,
    \quad \text{where}\quad 
    p=P_\eps v
    \quad \text{and}\quad 
    q=Q_\eps v.
\]
The fast variables of the system are $(u,q)$, while the retained slow component is $p$ and which will parametrize the slow manifold introduced later.  
Then, we define
\begin{equation}
    a_{\del,\eps}
    :=\min\set{\omega_f\del^{-1},\alpha_\eps},
    \qquad
    G_{\del,\eps}
    :=a_{\del,\eps}-\beta_\eps.
    \label{eq:combined-stable-rate}
\end{equation}
The quantity $G_{\del,\eps}$ is the separation between the slowest
part of the fast modes and the fastest retained slow mode.

\subsection{The renormalized slow manifold}

\begin{theorem}[Two-parameter slow manifold]
\label{thm:two-parameter-manifold}
Assume Assumptions \ref{ass:operator}, \ref{ass:spectral_shift}, \ref{ass:fast_operator} and \ref{ass:coupling}.  
Suppose that
\begin{equation}
    G_{\del,\eps}>0
    \quad\text{and}\quad
    K_{\del,\eps}
    :=\frac{6L_*}{G_{\del,\eps}}<1.
    \label{eq:two-parameter-gap}
\end{equation}
Then, the renormalized system \eqref{eq:renormalised-fast-slow} possesses a finite-dimensional invariant Lipschitz graph, called the slow manifold, given by
\begin{equation}
    \Mcal_{\del,\eps}
    :=\set{
      J_{\del,\eps}(p):
      p\in P_\eps\Hcal
    },
    \label{eq:two-parameter-graph}
\end{equation}
where
\begin{equation}
    J_{\del,\eps}(p)
    :=\left(
       h^u_{\del,\eps}(p),
       p+h^q_{\del,\eps}(p)
    \right)
    \label{eq:two-parameter-embedding}
\end{equation}
and
\begin{equation}
    h^u_{\del,\eps}:P_\eps\Hcal\to\Xcal,
    \qquad
    h^q_{\del,\eps}:P_\eps\Hcal\to Q_\eps\Hcal.
\end{equation}
The graph components satisfy
\begin{equation}
    \norm{h^u_{\del,\eps}}_{L^\infty}
    \leq\frac{M_f\del}{\omega_f},
    \qquad
    \norm{h^q_{\del,\eps}}_{L^\infty}
    \leq\frac{M_g}{\alpha_\eps}
    \leq\frac{M_g\eps^2}{\mu+\eps}.
    \label{eq:two-parameter-graph-size}
\end{equation}
Moreover, 
\begin{equation}
    \Lip(h^u_{\del,\eps})
    +\Lip(h^q_{\del,\eps})
    \leq
    \frac{4L_*}{(1-K_{\del,\eps})G_{\del,\eps}}.
    \label{eq:two-parameter-graph-lipschitz}
\end{equation}
Define the associated graph retraction by
\begin{equation}
    \mathcal R_{\del,\eps}(u,v)
    :=J_{\del,\eps}(P_\eps v).
    \label{eq:two-parameter-retraction}
\end{equation}
Then, the dynamics on the slow manifold $\Mcal_{\del,\eps}$ are conjugate through $J_{\del,\eps}$ to the finite-dimensional equation
\begin{equation}
    \dot p+A_sp
    =P_\eps g\left(
      h^u_{\del,\eps}(p),
      p+h^q_{\del,\eps}(p)
    \right).
    \label{eq:two-parameter-reduced}
\end{equation}

\end{theorem}

\begin{proof}
The argument follows the proof of Theorem \ref{thm:inertial-manifold} and the slow-manifold constructions in \cite[Prop. 5.2 - Prop. 5.8]{HummelKuehn2022} and \cite[Thm. 4.1]{KuehnSulzbach2025}.
\end{proof}

Write $\mathbb S_\del(t)$ and $\mathbb S_{\del,\eps}(t)$ for the original and renormalized semiflows, and set
\[
    \Phi_{\del,\eps}(t)=\mathbb S_{\del,\eps}(t)\big|_{\Mcal_{\del,\eps}}.
\]

\begin{remark}
\label{cor:joint-scaling}
If $0<\del\leq\vartheta\eps^2$ with $0<\vartheta<\omega_f/\mu$, then the spectral gap in the renormalized system satisfies
\[
    G_{\del,\eps}\geq c_\vartheta\eps^{-1},\quad \text{where}\quad 
    c_\vartheta=\min\{1,\omega_f/\vartheta-\mu\}>0.
\]
Indeed, the physical separation is at least $(\omega_f/\vartheta-\mu)\eps^{-2}$ and the shifted slow separation is at least $\eps^{-1}$.
Thus \eqref{eq:two-parameter-gap} holds for small $\eps$, with graph height $O(\del+\eps^2)$ and slope $O(\eps)$ along this parameter family.
\end{remark}

\subsection{Effect of the spectral renormalization}

Before proving the post-initial-layer reduction, we record the coupled analogue of the approximate semi-conjugacy comparison mechanism used in Section \ref{sec:post-layer}.

\begin{proposition}
\label{prop:coupled-full-flow-comparison}
Assume Assumptions~\ref{ass:operator}, \ref{ass:spectral_shift}, \ref{ass:fast_operator} and~\ref{ass:coupling}.
For $T>0$, set
\begin{equation}
    C_T^{\rm flow}=\e^{2L_*T}.
    \label{eq:full-flow-comparison-constant}
\end{equation}
For every $z_0=(u_0,v_0)\in\Xcal\times\Hcal$ and $0<\del,\eps\leq1$,
\begin{equation}
    \sup_{0\leq t\leq T}
    \norm{\mathbb S_\del(t)z_0-\mathbb S_{\del,\eps}(t)z_0}_{\Xcal\times\Hcal}
    \leq\frac{C_T^{\rm flow}}{\eps\lambda_\eps^+}
       \left(\norm{Q_\eps v_0}_{\Hcal}+\frac{M_g}{\alpha_\eps}\right).
    \label{eq:coupled-full-flow-comparison}
\end{equation}
\end{proposition}

\begin{proof}
Let $\widetilde q(t)=Q_\eps v^{\del,\eps}(t)$.
Variation of constants gives
\begin{equation}
    \sup_{t\geq0}\norm{\widetilde q(t)}_{\Hcal}
    \leq\norm{Q_\eps v_0}_{\Hcal}+\frac{M_g}{\alpha_\eps}.
    \label{eq:renormalized-q-bound-coupled}
\end{equation}
Write the difference using the original block generator $\operatorname{diag}(\del^{-1}A_f,A_s)$.
Its semigroup is contractive, and the additional forcing is $(0,\eps^{-1}\widetilde q)$ in the high slow modes.
Thus, for $D(t)=\norm{\mathbb S_\del(t)z_0-\mathbb S_{\del,\eps}(t)z_0}_{\Xcal\times\Hcal}$,
\[
\begin{aligned}
    D(t)&\leq2L_*\int_0^tD(s)\,ds
       +\eps^{-1}\int_0^t\e^{-\lambda_\eps^+(t-s)}\norm{\widetilde q(s)}\,ds\\
    &\leq2L_*\int_0^tD(s)\,ds
       +\frac1{\eps\lambda_\eps^+}
          \left(\norm{Q_\eps v_0}+\frac{M_g}{\alpha_\eps}\right).
\end{aligned}
\]
Gronwall's inequality proves the claim.
\end{proof}
The retraction \eqref{eq:two-parameter-retraction} preserves the retained slow coordinate exactly and replaces the two stable components by their graph values.

For $\kappa>0$ define the two initial-layer scales
\begin{equation}
    \tau_{\del,\eps}^{\kappa}
    :=\max\left\{
      \frac{\kappa\del}{\omega_f}\abs{\log\del},
      \frac{\kappa}{\mu}\eps^2\abs{\log\eps}
    \right\}.
    \label{eq:two-layer-time}
\end{equation}
The first term damps the physical fast variable and the second one is
exactly the high-mode layer from \eqref{eq:layer-time}.

\begin{theorem}[Two-parameter post-initial-layer shadowing]
\label{thm:two-parameter-shadowing}
Assume the hypotheses of Theorem \ref{thm:two-parameter-manifold}.
Let $\Bcal\subset\Xcal\times\Hcal$ be bounded and set
\begin{equation}
    R_u:=\sup_{(u_0,v_0)\in\Bcal}\norm{u_0}_{\Xcal},
    \qquad
    R_v:=\sup_{(u_0,v_0)\in\Bcal}\norm{v_0}_{\Hcal}.
    \label{eq:two-parameter-radii}
\end{equation}
Then, for every $T>0$ and $\kappa>0$, all sufficiently small
parameters for which $\tau_{\del,\eps}^{\kappa}<T$ satisfy
\begin{equation}
\begin{aligned}
    &\sup_{z_0\in\Bcal}
      \sup_{\tau_{\del,\eps}^{\kappa}\leq t\leq T}
      \norm{
       \mathbb S_\del(t)z_0
       -\Phi_{\del,\eps}
          (t-\tau_{\del,\eps}^{\kappa})
        \mathcal R_{\del,\eps}
          \mathbb S_\del(\tau_{\del,\eps}^{\kappa})z_0
      }_{\Xcal\times\Hcal}\leq
      C_T\left(
        R_u\del^{\kappa}
        +R_v\eps^{\kappa}
        +\del+\eps^2
      \right),
    \label{eq:two-parameter-shadowing}
\end{aligned}
\end{equation}
where $C_T$ is independent of sufficiently small $\del,\eps$ in the
chosen parameter family.
\end{theorem}

\begin{proof}
Set $\tau:=\tau_{\del,\eps}^{\kappa}$ and $z_\tau:=\mathbb S_\del(\tau)z_0$.
Variation of constants gives, uniformly for $z_0\in\Bcal$,
\[
    \norm{u_\tau}\leq R_u\del^\kappa+\frac{M_f\del}{\omega_f},\qquad
    \norm{Q_\eps v_\tau}\leq R_v\eps^\kappa+\frac{M_g}{\lambda_\eps^+}.
\]
Together with the graph bounds, this yields
\[
    \norm{z_\tau-\mathcal R_{\del,\eps}z_\tau}
    \leq C(R_u\del^\kappa+R_v\eps^\kappa+\del+\eps^2).
\]
The renormalized flow has Lipschitz constant at most $\e^{2L_*s}$ at time $s$.
By graph invariance and Proposition~\ref{prop:coupled-full-flow-comparison}, for $0\leq s\leq T-\tau$,
\[
\begin{aligned}
    \norm{\mathbb S_\del(s)z_\tau-
          \Phi_{\del,\eps}(s)\mathcal R_{\del,\eps}z_\tau}
    &\leq\e^{2L_*T}\left[
       \norm{z_\tau-\mathcal R_{\del,\eps}z_\tau}
       +\frac{\norm{Q_\eps v_\tau}+M_g/\alpha_\eps}{\eps\lambda_\eps^+}
    \right].
\end{aligned}
\]
The last fraction is $O(R_v\eps^{\kappa+1}+\eps^3)$.
Use $0<\eps\leq1$, put $s=t-\tau$, and apply the semiflow property to obtain the assertion.
\end{proof}

\begin{remark}
\label{rem:two-parameter-post-layer-structure}
The main difference to the post-initial layer result is that there are now two stable directions.  
After the time $\tau_{\del,\eps}^\kappa$, the physical fast variable has size $O(\del^\kappa+\del)$ and the unresolved slow tail has size $O(\eps^\kappa+\eps^2)$.  
Consequently the evolved state is already within
\[
    O(\del^\kappa+\eps^\kappa+\del+\eps^2)
\]
of the renormalized slow manifold.

For $s=t-\tau$, the comparison used in the proof is simply
\[
\begin{aligned}
\mathbb S_\delta(s)z_\tau
-\Phi_{\delta,\eps}(s)\mathcal R_{\delta,\eps}z_\tau
={}&
\bigl[
\mathbb S_\delta(s)z_\tau
-\mathbb S_{\delta,\eps}(s)z_\tau
\bigr] +
\bigl[
\mathbb S_{\delta,\eps}(s)z_\tau
-\mathbb S_{\delta,\eps}(s)\mathcal R_{\delta,\eps}z_\tau
\bigr],
\end{aligned}
\]
where graph invariance gives
\[
\mathbb S_{\delta,\eps}(s)\mathcal R_{\delta,\eps}z_\tau
=
\Phi_{\delta,\eps}(s)\mathcal R_{\delta,\eps}z_\tau.
\]
The first term measures the effect of the spectral renormalization
and is $O(R_v\eps^{\kappa+1}+\eps^3)$.
The second is controlled by finite-time Lipschitz dependence of the
renormalized semiflow and the post-layer distance
\[
\|z_\tau-\mathcal R_{\delta,\eps}z_\tau\|
=
O(R_u\delta^\kappa+R_v\eps^\kappa+\delta+\eps^2).
\]
\end{remark}


\section{Discussion}
\label{sec:discussion}

Here, we now discuss several aspects of the spectral renormalization method introduced and analyzed in the previous sections.
We remark, in particular, on the freedom in choosing the parameters in the spectral modification, its relation to other reduction methods, and the limitations and possible extensions of the dynamical conclusions.

\paragraph{Position and size of the artificial spectral gap.}
We note that the cutoff and damping strength can be chosen independently:
\[
    A_{\Lambda,\kappa}=A+\kappa Q_\Lambda,
    \qquad Q_\Lambda=\mathbf1_{(\Lambda,\infty)}(A),
    \qquad \Lambda,\kappa>0.
\]
The separation between retained and discarded spectra is at least $\kappa$, so the graph construction in Theorem~\ref{thm:inertial-manifold} applies when $\kappa\geq8L$.
For approximation, the spectral-tail estimate gives
\[
    \norm{\kappa Q_\Lambda}_{\mathcal L(\Hcal^r,\Hcal^s)}
    \leq\kappa(1+\Lambda)^{-(r-s)/2},\qquad r>s\geq0.
\]
Thus a fixed damping strength can be placed progressively farther into the spectrum while its norm between these spaces tends to zero.
More generally, for $\Lambda_\eps\sim\eps^{-b}$ and $\kappa_\eps\sim\eps^{-a}$, with $b>0$ and $a\geq0$, the condition $b(r-s)/2>a$ guarantees this convergence.
This separates the gap needed for the graph construction from the regularity and scaling needed for approximation.
In the fast--slow setting, the physical separation condition of Section~\ref{sec:fast-slow} must also be satisfied.

\paragraph{Other regularization and transformation mechanisms.}
One example is hyperviscosity. 
Here one replaces the spectral renormalization of $A$, for example, by $A+\nu A^\theta$, with $\nu>0$ and $\theta>1$.
It strengthens high-frequency decay but changes the operator domain; corresponding approximation estimates therefore require different regularity bounds.
Inertial-manifold results for hyperviscous models can also use spatial averaging; see the results in \cite{KostiankoEtAl2022}.

Another example is introduced by Chepyzhov, Kostianko and Zelik \cite{ChepyzhovKostiankoZelik2019}, where they study the hyperbolic relaxation 
\begin{equation*}
    \eta\partial_t^2u+\partial_tu+Au=F(u),\qquad\eta>0,
    \label{eq:hyperbolic-relaxation-discussion}
\end{equation*}
their parabolic limit and the existence of inertial manifolds.
In their construction they retain the classical condition $\lambda_{N+1}-\lambda_N>2L$ and additionally require $\eta(3\lambda_{N+1}+\lambda_N)\leq1$.
Thus small hyperbolic relaxation does not itself remove the spatial-gap requirement.

Nonlocal changes of variables provide another route for the classes of one-dimensional reaction-diffusion-advection equations treated in \cite{KostiankoZelik2017,KostiankoZelik2018}.
On the region where the transformation is a diffeomorphism, it preserves trajectories in changed coordinates.
The periodic vector-valued case also contains counterexamples to inertial-manifold existence, so this mechanism depends on the equation and boundary conditions.

\paragraph{Relation with shadowing and structural stability.}
In a recent article Arrieta, Carvalho and Takaessu \cite{ArrietaCarvalhoTakaessu2026} prove Lipschitz shadowing for the time-one map on the global attractor and H\"older shadowing on positively invariant bounded neighborhoods, under smoothness and Morse-Smale hypotheses with only equilibria in the non-wandering set.
Their results do not assume an inertial manifold.
However, combining this theory with our one-step comparison estimates on complete bounded trajectories is a possible route to long-time orbit comparison beyond the contraction settings considered in this work.
The shadowing orbit may require a different initial point, and H\"older shadowing may yield a weaker power of $\eps$.
Such an orbit comparison would therefore not directly establish the uniform defect bound for the prescribed retraction $R_\eps$ required in Theorem~\ref{thm:hausdorff-from-uniform-semiconjugacy}.

\paragraph{Computational inertial manifolds.}
The present construction also provides a possible interpretation of the long-standing computational success of approximate inertial manifolds, nonlinear Galerkin schemes, and post-processed Galerkin methods \cite{FoiasJollyKevrekidisSellTiti1988,JollyKevrekidisTiti1990, MarionTemam1989,DebusscheMarion1992,GarciaArchillaNovoTiti1998, DevulderMarion1992,KoronakiEtAl2024}. 
These methods exploit the observation that, after the rapidly decaying modes have relaxed, the unresolved high-frequency component is approximately slaved to a finite set of retained coordinates. 
Such procedures can remain useful even in equations for which the existence of an exact inertial manifold has not been established. 
The renormalization viewpoint suggests one explanation for this phenomenon: although the original equation may fail the spectral-gap condition, it may lie close, on the dynamically relevant region and time scale, to a renormalized equation possessing an exact inertial manifold. 
The graph of this nearby manifold then provides a high-mode slaving relation for the chosen renormalization and preparation, while the semi-conjugacy and post-initial-layer estimates quantify how accurately its reduced dynamics reproduce the original equation. 
This does not by itself prove convergence of a particular computational inertial-manifold algorithm, but it provides a nearby equation interpretation which may help explain why low-dimensional nonlinear reductions can remain effective even when an exact inertial manifold for the unmodified equation is unavailable.

\paragraph{Approximate invariant manifolds for the original equation.}
The renormalized inertial manifold $\Mcal_\eps$ can also be interpreted as an approximate invariant manifold for the original equation. 
This connects our present construction with the classical theory of approximate inertial manifolds, where finite-dimensional graphs are used to approximate the long-time dynamics of dissipative PDEs without requiring exact invariance; see, for example, 
\cite{FoiasManleyTemam1988,Marion1989,Temam1989,Titi1990, DebusscheMarion1992,DebusscheTemam1994}.
Here the interpretation follows directly from the semiflow estimates established above. 
Indeed, trajectories of the original system starting on $\Mcal_\eps$ satisfy
\[
    \sup_{u_0\in\Mcal_\eps}\sup_{t\geq0}
    \dist\bigl(S(t)u_0,\Mcal_\eps\bigr)
    \leq C\eps^2,
\]
while, on every fixed time interval, the corresponding dynamics are approximated by the exact flow on the renormalized inertial manifold with the same order. 
More generally, Theorem \ref{thm:post-layer} shows that trajectories starting in bounded sets are $O(\eps^2)$-shadowed by trajectories on $\Mcal_\eps$ after an initial layer of length
$O(\eps^2|\log\eps|)$. 
Thus $\Mcal_\eps$ provides an approximate high-mode slaving relation for the original equation even when the latter does not possess an exact inertial manifold.

This viewpoint is related to, but weaker than, the persistence theory of Bates, Lu and Zeng \cite{BatesLuZeng1998,BatesLuZeng2008}, which requires differentiable approximate invariance together with suitable normal hyperbolicity in order to obtain a genuine nearby invariant manifold.
Such hypotheses need not survive removal of the artificial spectral gap.
The advantage of the present semi-conjugacy approach is precisely that it still provides a quantitative finite-dimensional approximation of the dynamics in this regime.

\section*{Acknowledgments}

J.E.S. acknowledges funding from the Deutsche Forschungsgemeinschaft (DFG, German Research Foundation) – Project-ID 571660837.

\paragraph{Use of generative AI.} 
Generative AI tools (ChatGPT 5.6 Sol and 6 Astra) were used during the preparation of this manuscript, specifically for brainstorming ideas, additional literature review, and draft editing. 
All mathematical results and final formulations are the authors’ own work.

\small
\bibliographystyle{abbrvurl}
\bibliography{literature}

\end{document}